\documentclass[a4paper,12pt,fleqn]{article}

\usepackage[T1]{fontenc}	
\usepackage{lmodern}	
\usepackage[final,activate,verbose=true,auto=true]{microtype}  

\usepackage{amsmath, amssymb, amstext, amsthm}  
\usepackage{mathrsfs} 
\usepackage{graphicx}

\usepackage[textheight=235mm,textwidth=160mm]{geometry} 
\newtheorem{thm}{Theorem}
\newtheorem{lem}{Lemma}
\newtheorem{cor}{Corollary}
\newtheorem{ass}{Assumption}
\theoremstyle{remark}
\newtheorem{rem}{Remark}
\theoremstyle{definition}
\newtheorem{expl}{Example}

\def\ow{o\kern-.42em\raise.82ex\hbox{
\vrule width .12em height .0ex depth .075ex \kern-0.16em \char'56}\kern-.07em}
\def\OW{O\kern-.460em\raise1.36ex\hbox{
\vrule width .13em height .0ex depth .075ex \kern-0.16em \char'56}\kern-.07em}

\title{\bf 
Semidefinite approximations of the polynomial abscissa$^0$
}

\begin{document}

\author{Roxana He{\ss}$^{1,2}$
Didier Henrion$^{1,2,3}$\\ Jean-Bernard Lasserre$^{1,2,4}$ Ti\'{\^e}n-S\ow n Ph\d{a}m$^5$}

\footnotetext[0]{A part of this work was done while the fourth author was visiting LAAS-CNRS in April 2015. He would like to thank LAAS-CNRS and J. B. Lasserre for the hospitality and support during his stay.}
\footnotetext[1]{CNRS; LAAS; 7 avenue du colonel Roche, F-31400 Toulouse; France.}
\footnotetext[2]{Universit\'e de Toulouse; LAAS; F-31400 Toulouse; France}
\footnotetext[3]{Faculty of Electrical Engineering, Czech Technical University in Prague,
Technick\'a 2, CZ-16626 Prague, Czech Republic}
\footnotetext[4]{Institut de Math\'ematiques de Toulouse, Universit\'e de
Toulouse; UPS; F-31062 Toulouse, France.}
\footnotetext[5]{Department of Mathematics, University of Dalat, 1 Phu Dong Thien Vuong, Dalat, Vietnam.}

\date{Draft of \today}

\maketitle

\begin{abstract}
Given a univariate polynomial, its abscissa is the maximum real part of its roots. The abscissa arises naturally
when controlling linear differential equations. As a function of the polynomial coefficients,
the abscissa is H\"older continuous, and not locally Lipschitz in general, which is a source of numerical
difficulties for designing and optimizing control laws.
In this paper we propose simple approximations of the abscissa given by polynomials
of fixed degree, and hence controlled complexity. Our approximations are computed by a hierarchy
of finite-dimensional convex semidefinite programming problems. When their degree tends to infinity,
the polynomial approximations converge in norm to the abcissa, either from above or from below.\\
\begin{center}
{\bf Keywords}\\[1em]
Linear systems control, non-convex non-smooth optimization, polynomial approximations,
semialgebraic optimization, semidefinite programming.
\end{center}
\end{abstract}

\section{Introduction}

Given a univariate polynomial, its abscissa is the maximum real part of its roots. When studying
linear differential equations, the abscissa of the characteristic polynomial of the equation is used as a measure
of the decay or growth rate of the solution. In linear systems control, the abscissa
function is typically parametrized by a small number of real parameters (the controller coefficients), and it should
be minimized so as to ensure a sufficiently fast decay rate of closed-loop trajectories.

As a function of
the polynomial coefficients (expressed in some basis), the abscissa is a H\"older continuous function (with exponent
equal to the reciprocal of the polynomial degree), but it is not locally Lipschitz.
As a consequence of this low regularity, numerical optimization of the polynomial abscissa is typically a challenge.

For a recent survey on the abscissa
function and its applications in systems control, see \cite{cross}. A detailed variational analysis of the abscissa was
first carried out in \cite{gausslucas_pisa}. These ideas were exploited in a systems control setup in \cite{bhlo},
using randomized techniques of non-convex non-smooth local optimization, however without rigourous
convergence guarantees.

In the space of controller parameters,
the zero sublevel set of the abscissa function of the characteristic polynomial of a linear system is the so-called
stabilizability region, and it is typically non-convex and non-smooth, see \cite{stability}
where this set is approximated with simpler sets such as balls or ellipsoids.  In \cite{innerpmi}, ellipsoidal approximations
of the stabilizability region were generalized to polynomial sublevel set approximations, obtained by replacing negativity of
the abscissa function with positive definiteness of the Hermite matrix of the characteristic polynomial.

This paper continues the research efforts of \cite{stability} and \cite{innerpmi}, in the sense that we would like to approximate
the complicated geometry of the abscissa function (and its sublevel sets) with a simpler function, namely a
low degree polynomial. The level of complexity of the approximation is the degree of the polynomial,
to be fixed in advance. Moreover, we would like the quality of the approximation to improve when the
degree increases, eventually converging (in some appropriate sense) to the original abscissa function when
the degree tends to infinity. 

The outline of the paper is as follows. After introducing in Section 2 the abscissa function and some relevant notations, we address in Section 3  the problem of finding an upper approximation of the abscissa. In Section 4, we address the more difficult problem of approximating the abscissa from below, first by using elementary symmetric functions, and second by using the Gau{\ss}-Lucas theorem, inspired by \cite{gausslucas_pisa}. Explicit numerical examples illustrate
our findings throughout the text.

\section{Preliminaries}

\subsection*{Notation and definitions}

Let $n \in \mathbb{N}$ and $\mathcal{Q} \subseteq \mathbb{R}^n$ be a compact semi-algebraic set on which a Borel measure with support $\mathcal{Q}$ can be defined and whose moments are easy to compute.  For simplicity, in this paper we choose $\mathcal{Q}=[-1,1]^n=\{q \in \mathbb{R}^n \: :\: 1-q^2_1\geq 0, \ldots, 1-q^2_n\geq 0\}$.

Let $\mathscr{C}(\mathcal{Q})$ denote the space of continuous functions on $\mathcal{Q}$. Its topological dual 
is isometrically isomorphic to the vector space
$\mathscr{M}(\mathcal{Q})$ of signed Borel measures on $\mathcal{Q}$.
By Banach-Alaoglu' s theorem \cite{ash,barvinok}, the unit ball of $\mathscr{M}(\mathcal{Q})$ 
is compact (and sequentially compact) in the weak-star topology of $\mathscr{M}(\mathcal{Q})$.

Denote by $\mathbb{R}[q]_d$ the vector space of real polynomials in the variables $q = (q_1, \dotsc, q_n)$ of degree at most $d$. Let $\Sigma[q] \subset \mathbb{R}[q]$ be the convex cone of real polynomials that are sums of squares of polynomials and denote by $\Sigma[q]_{2d}$ its subcone of sum of squares polynomials of degree at most $2d$.

\subsection*{The abscissa function}

Consider the monic non-constant polynomial $p \in \mathbb{R}[s]$ defined by
\begin{equation*}
	p:\ s \mapsto p(q,s) := \sum_{k=0}^m p_k(q)s^k
\end{equation*}
with $s \in \mathbb{C}$ complex, $q = (q_1,\dotsc,q_n) \in \mathcal{Q}$ and given polynomials $p_k \in \mathbb{R}[q]$ for $k = 0,1, \dotsc, m$ with $p_m(q) \equiv 1$ and $m > 0$.

Denote by $s_k(q),\ k = 1, \dotsc, m,$ the roots of $p(q,\cdot)$ and by $a: \mathcal{Q}  \to \mathbb{R}$ (or $a_p$ if it is necessary to clarify the dependence on the polynomial) the abscissa map of $p$, i.e. the maximal real part of the roots:
\begin{equation*}
	a(q) := \max_{k = 1, \dotsc, m} \Re(s_k(q)), \:\: q\in\mathcal{Q}.
\end{equation*}

Equivalently, with $i = \sqrt{-1}$ and $s = x + iy$ write
\begin{equation*}
	p(q,s) = p_\Re(q,x,y) + ip_\Im(q,x,y)
\end{equation*}
for two real polynomials $p_\Re, p_\Im \in \mathbb{R}[q,x,y]$ of total degree $m$. Then
\begin{equation*}
	a(q) = \max\{x \in \mathbb{R} : \exists y \in \mathbb{R}:\ p_\Re(q,x,y) = p_\Im(q,x,y) = 0\},\:\: q\in\mathcal{Q}.
\end{equation*}

We observe that function $a:\mathcal{Q}\to\mathbb{R}$ is semi-algebraic and we define the basic closed semi-algebraic set
\begin{equation*}
	\mathcal{Z} := \{(q,x,y) \in \mathbb{R}^n \times \mathbb{R}^2 :  q \in \mathcal{Q},\: p_\Re(q,x,y) = p_\Im(q,x,y) = 0\}.
\end{equation*}

\begin{rem}
Set $\mathcal{Z}$ is compact, since $\mathcal{Q}$ is compact and $p$ is monic in $s$.
\end{rem}

Now we can write the abscissa map as
\begin{equation*}
	a(q) = \max\{x \in \mathbb{R} : \exists y \in \mathbb{R}:\ (q,x,y) \in \mathcal{Z}\},\:\: q\in\mathcal{Q}.
\end{equation*}

Since $p$ is monic, its abscissa $a$ is continuous, though in general not Lipschitz continuous. For example, for $n = 1$ and $p(q,s) = s^6 + q$ the map $a(q)$ is only H\"older continuous with exponent $\frac{1}{6}$ for small $q$. To be precise, $a$ is always H\"older continuous by the {\L}ojasiewicz inequality \cite{bochnak}, since $\mathcal{Q}$ is compact.
\section{Upper abscissa approximation}\label{upper}

\subsection{Primal and dual formulation}\label{sec1}

Given a polynomial $p$ defined as above, the solution to the following linear programming (LP) problem gives an upper approximation of its abscissa function $a$ on $\mathcal{Q}$:
\begin{align}
	\rho = & \inf_{v \in \mathscr{C}(\mathcal{Q})}  \int_\mathcal{Q} v(q)\, dq \label{v}\\
	& \text{s.t.}\quad v(q) - x \geq 0 \text{ for all } (q,x,y) \in \mathcal{Z} \notag
\end{align}
with $\mathscr{C}(\mathcal{Q})$ denoting the space of continuous functions from $\mathcal{Q}$ to $\mathbb{R}$.

\begin{rem}
	Since the continuous functions defined on compact set $\mathcal{Q}$ can be approximated uniformly by polynomials by the Stone-Weierstra{\ss} theorem \cite[\S 16.4.3]{zorich}, we can replace $\mathscr{C}(\mathcal{Q})$ in problem \eqref{v} by the ring of polynomials $\mathbb{R}[q]$.
\end{rem}

The LP dual to problem \eqref{v} reads
\begin{align}
	\rho^* = & \sup_{\mu \in \mathscr{M}^+(\mathcal{Z})}  \int_\mathcal{Z} x\ d\mu(q,x,y) \label{mu}\\
	& \text{s.t.}  \int_\mathcal{Z} q^\alpha\, d\mu = \int_\mathcal{Q} q^\alpha\, dq, \ \text{ for all } \alpha \in \mathbb{N}^n, \notag
\end{align}
where $q^\alpha$ stands for the monomial $q_1^{\alpha_1} q_2^{\alpha_2} \cdots q_n^{\alpha_n}$ and $\mathscr{M}^+(\mathcal{Z})$ denotes the cone of non-negative Borel measures supported on $\mathcal{Z}$,
identified with  the set of all non-negative continuous linear functionals acting on $\mathscr{C}^+(\mathcal{Z})$, the
cone of non-negative continuous functions supported on $\mathcal{Z}$.

\begin{rem}\label{marginal}
	The constraint $\int_\mathcal{Z}  q^\alpha\, d\mu = \int_\mathcal{Q} q^\alpha\, dq$ for all $\alpha \in \mathbb{N}^n$ implies that the marginal of $\mu$ on $\mathcal{Q}$ is the Lebesgue measure on $\mathcal{Q}$, i.e. for every $g \in \mathscr{C}(\mathcal{Q})$ it holds that
	\begin{equation*}
		\int_\mathcal{Z} g(q)\, d\mu(q,x,y) = \int_\mathcal{Q} g(q)\, dq.
	\end{equation*}
	In particular this implies that $\Vert\mu\Vert={\rm vol}\:\mathcal{Q}$ where ${\rm vol}(\cdot)$ denotes the volume or Lebesgue measure.
\end{rem}

\begin{lem}\label{zero}
The supremum in LP \eqref{mu} is attained, and there is no duality gap between LP (\ref{v}) and LP (\ref{mu}), i.e. $\rho=\rho^*$.
\end{lem}

\begin{proof}
	The set of feasible solutions for the dual LP \eqref{mu} is a bounded subset of $\mathscr{M}^+(\mathcal{Z})$ with $\mathcal{Z}$ compact and therefore it is weak-star compact. Since the objective function is linear, its supremum on this weak-star compact set is attained. For elementary background on weak-star topology, see e.g. \cite[Chapter IV]{barvinok}.

	To prove that there is no duality gap, we apply \cite[Theorem IV.7.2]{barvinok}. For this purpose we introduce the notation used in \cite{barvinok} in this context. There, the primal and the dual are written in the following canonical form:
	\begin{align*}
		\rho^* = &\sup_{{\bf x} \in E_1}\ \langle {\bf x},{\bf c} \rangle_1 & \rho = &\inf_{{\bf y} \in F_2}\ \langle {\bf b},{\bf y} \rangle_2\\
		& \text{s.t. } {\bf Ax} = {\bf b},\ {\bf x} \in E_1^+ & & \text{s.t. } {\bf A}^*{\bf y} - {\bf c} \in F_1^+
	\end{align*}	
	So we set $E_1 := \mathscr{M}(\mathcal{Z})$ with its cone $E_1^+ := \mathscr{M}^+(\mathcal{Z})$. Then their (pre-)duals are $F_1 := \mathscr{C}(\mathcal{Z})$ and $F_1^+ := \mathscr{C}^+(\mathcal{Z})$ respectively. Similarly, we define $E_2 := \mathscr{M}(\mathcal{Q})$ and $F_2 := \mathscr{C}(\mathcal{Q})$.\par\smallskip
	Setting ${\bf x} := \mu \in E_1$, ${\bf c} := x \in F_1$, ${\bf b} \in E_2$ the Lebesgue measure on $\mathcal{Q}$ and ${\bf y} := v \in F_2$, the linear operator ${\bf A}: E_1 \to E_2$ is given by ${\bf x} \mapsto \pi_{\mathcal{Q}}{\bf x}$ where $\pi_\mathcal{Q}$ denotes the projection onto $\mathcal{Q}$, i.e.,
	${\bf Ax}(B)={\bf x}(B\times \mathbb{R}^2)$ for all $B\in\mathcal{B}(\mathcal{Q})$.

	According to \cite[Theorem IV.7.2]{barvinok} the duality gap is zero if the cone $\{({\bf Ax}, \langle {\bf x},{\bf c} \rangle_1) : {\bf x} \in E_1^+\}$
		is closed in $E_2 \times \mathbb{R}$. This holds in our setup since ${\bf x} \mapsto {\bf Ax}$ and ${\bf x} \mapsto \langle {\bf x},{\bf c} \rangle_1$ are continuous linear maps and $E_1^+ = \mathscr{M}^+(\mathcal{Z})$ is weak-star closed due to the compactness of $\mathcal{Z}$. So if for some ${\bf a}\in E_2$, ${\bf Ax_n}\to {\bf a}$ as $n\to\infty$
		then from the definition of the mapping ${\bf A}$ and as $({\bf x_n})\subset E_1^+$, one has $\Vert{\bf x_n}\Vert\to \Vert {\bf a}\Vert$ as $n\to\infty$ (see Remark \ref{marginal}). Therefore the sequence $({\bf x_n})\subset E_1^+$ is bounded and by Banach-Alaoglu's theorem  \cite{ash,barvinok}, it contains a subsequence $({\bf x_{n_k}})\subset E^+_1$ that converges to some ${\bf x}\in E^+_1$ for the weak-star topology. By continuity of the mappings
		${\bf A}$ and ${\bf c}$, the result follows.
\end{proof}

\begin{rem}\label{attained}
	The infimum in LP \eqref{v} is not necessarily attained, since the set of feasible solutions is not compact. It is neither attained when we replace $\mathscr{C}(\mathcal{Q})$ with $\mathbb{R}[q]$, since $a$ is non-Lipschitz, so in particular not polynomial.

	However, the infimum is attained if we replace $\mathscr{C}(\mathcal{Q})$ with $\mathbb{R}[q]_d$ for $d$ finite. Then, with $M := \min_{q \in \mathcal{Q}} a(q) > -\infty$ and $\tilde{v}(q) := v(q) - M$ we can rewrite LP \eqref{v} as the equivalent problem
	\begin{equation*}
		\inf_{\tilde{v} \in \mathbb{R}[q]_d} \int_\mathcal{Q} \tilde{v}(q)dq\ \text{ s.t. } \tilde{v}(q) + M - x \geq 0 \text{ on } \mathcal{Z}.
	\end{equation*}
	Now, any feasible $\tilde{v}$ is non-negative on $\mathcal{Q}$, so $\int_\mathcal{Q} \tilde{v}(q)dq = \|\tilde{v}\|_{L^1} \geq 0$ and for every $R \in \mathbb{R}$ the set $\{\tilde{v} \in \mathbb{R}[q]_d \: :\: R \geq \int_\mathcal{Q} \tilde{v}(q)dq$ and $\tilde{v}(q) + M - x \geq 0 \text{ on } \mathcal{Z}\}$ is closed and bounded in the strong topology, thus compact. Besides, due to the continuity of $a$, there always exists an $R < \infty$ such that the mentioned set is not empty, hence the infimum is attained.
\end{rem}

\subsection{SDP hierarchy}\label{las}

Let $d_0 \in \mathbb{N}$ be sufficiently large. As presented in \cite{lasserre}, we can write a hierarchy of finite-dimensional
convex semidefinite programming (SDP) problems for LP \eqref{v} indexed by the parameter $d \in \mathbb{N},\ d \geq d_0$:
	\begin{align}
		\rho_d = & \inf_{v_d,\sigma_0,\sigma_{j},\tau_\Re,\tau_\Im} \int_\mathcal{Q} v_d(q)\, dq\notag\\
		& \text{s.t.}\ v_d(q) - x = \sigma_0(q,x,y) + \sum_{j=1}^n \sigma_{j}(q,x,y)(1-q_j^2)\label{vd}\\
		&\hspace{2,5cm} + \tau_\Re(q,x,y)p_\Re(q,x,y) + \tau_\Im(q,x,y)p_\Im(q,x,y)\notag
	\end{align}
for all $(q,x,y) \in \mathbb{R}^n \times \mathbb{R}^2$ and with $v_d \in \mathbb{R}[q]_{2d}$, $\sigma_0 \in \Sigma[q,x,y]_{2d},\ \sigma_{j} \in \Sigma[q,x,y]_{2d-2}$ for $j = 1, \dotsc, n$ and $\tau_\Re, \tau_\Im \in \mathbb{R}[q,x,y]_{2d-m}$.
\begin{rem}\label{qm}
	The quadratic module generated by the polynomials $1-q_1^2, \dotsc, 1-q_n^2,\ \pm p_\Re,\ \pm p_\Im$ is archimedean by \cite[Lemma 3.17]{laurent}, since it contains the polynomial $f(q,x,y) := \sum_{j=1}^n (1-q_j^2) - p_\Re^2 - p_\Im^2$ and the set $\{(q,x,y) \in \mathbb{R}^n \times \mathbb{R}^2 : f(q,x,y) \geq 0 \}$ is compact. By \cite[Theorem 4.1]{lasserre}, this implies that the hierarchy converges, i.e. $\lim_{d\to\infty} \rho_d = \rho$.
\end{rem}

\begin{rem}\label{strength}
Note that SDP \eqref{vd} is not equivalent to LP \eqref{v}, not even with $\mathscr{C}(\mathcal{Q})$ replaced by $\mathbb{R}[q]$ or $\mathbb{R}[q]_{2d}$ in the latter, but it is a strengthening of it, meaning $\rho_d \geq \rho$. To be more specific, SDP \eqref{vd} is a reinforcement of the following LP:
	\begin{equation*}
		 \inf_{v \in \mathbb{R}[q]_{2d}}  \int_\mathcal{Q} v(q)\, dq\ \text{ s.t. } v(q) - x > 0 \text{ for all } (q,x,y) \in \mathcal{Z}.
	\end{equation*}
where we exchanged non-negativity for a specific certificate of positivity. See \cite[Chapter 4.2]{lasserre} for details.
\end{rem}

\begin{expl}
	The infimum in SDP \eqref{vd} is not necessarily attained, e.g. consider the polynomial $p(q,s) = s^2$. Then $p_\Re(q,x,y) = x^2 - y^2$, $p_\Im(q,x,y) = 2xy$ and $\mathcal{Z} = \mathcal{Q} \times \{(0,0)\}$. Obviously, the optimal solution to LP \eqref{v} is $v \equiv 0$. For SDP \eqref{vd} we would want
		\begin{equation*}
			v(q) - x = \sigma_0(q,x,y) + \sigma_1(q,x,y)(1-q^2) + \tau_\Re(q,x,y)(x^2-y^2) + 2\tau_\Im(q,x,y)xy,
		\end{equation*}
	meaning $0 \equiv v = x + \sigma_0 + \sigma_1(1-k^2) + \tau_\Re x^2 - \tau_\Re y^2 + 2\tau_\Im xy$ with $\sigma_0, \sigma_1$ sums of squares. This is impossible, since it would require the construction of the term $-x$ which in this case is only possible as a summand of $\sigma_0$. Then however we would always also produce a constant positive term.	Practically this means that the multipliers $\sigma_0, \sigma_1, \tau_\Re, \tau_\Im$ blow up.
\end{expl}

Hence, an optimal solution might not exist, but we always have a near optimal solution. This means we should allow solutions $v_d$ with $\int_\mathcal{Q} v_d(q)\, dq \leq \rho_d + \tfrac{1}{d}$, e.g. in the above example we would search for  $v \equiv \varepsilon$ for an $\varepsilon > 0$ sufficiently small.

\begin{rem}
	The existence of an optimal solution depends on further conditions, like the ideal generated by the polynomials $1-q^2_j$, $p_\Re$ and $p_\Im$ being radical, and goes beyond the scope of this paper. The interested reader is referred to the proof of \cite[Lemma 1]{innerpmi} for further details.
\end{rem}

In the following theorem we prove that the associated sequence of solutions converges:

\begin{thm} \label{converge}
Let $v_d \in \mathbb{R}[q]_{2d}$ be a near optimal  solution for SDP \eqref{vd}, i.e.
$\int_\mathcal{Q} v_d(q)\, dq \leq \rho_d + \tfrac{1}{d}$, and consider the associated sequence $(v_d)_{d\geq d_0} \subset L^1(\mathcal{Q})$. Then $v_d$ converges to the abscissa $a$ in $L^1$ norm on $\mathcal{Q}$.
\end{thm}

\begin{proof}
	Recall that $\rho^*=\rho$ according to Lemma \ref{zero}. First we show that $\rho = \int_\mathcal{Q} a(q)\, dq$.
	For every $(q,x,y) \in \mathcal{Z}$ we have $x \leq a(q)$ and since $\int_\mathcal{Z} q^\alpha d\mu = \int_\mathcal{Q} q^\alpha dq$ for all $\alpha \in \mathbb{N}^n$ which means that the marginal of $\mu$ on $\mathcal{Q}$ is the Lebesgue measure on $\mathcal{Q}$ (see Remark \ref{marginal}), it follows that for every feasible solution $\mu \in \mathscr{M}_+(\mathcal{Z})$ it holds that
	\begin{equation*}
		\int_\mathcal{Z} x\ d\mu(q,x,y) \leq \int_\mathcal{Z} a(q)\, d\mu(q,x,y) = \int_\mathcal{Q} a(q)\, dq.
	\end{equation*}	
	Hence $\rho \leq \int_\mathcal{Q} a(q)\, dq$. On the other hand, for every $q \in \mathcal{Q}$ there exists $(q,x_q,y_q) \in \mathcal{Z}$ such that $a(q) = x_q$. Let $\mu^*$ be the Borel measure concentrated on $(q,x_q,y_q)$ for all $q \in \mathcal{Q}$, i.e. for $\mathcal{A}$ in the Borel sigma algebra of $\mathcal{Z}$ it holds
	\begin{equation*}
		\mu^*(\mathcal{A}) := \boldsymbol{1}_{\mathcal{A}}(q,x_q,y_q).
			\end{equation*}	
	Then $\mu^*$ is feasible for problem \eqref{mu} with value
	\begin{equation*}
		\int_\mathcal{Z} x\ d\mu^*(q,x,y) = \int_\mathcal{Q} a(q)\, dq,
	\end{equation*}	
	which proves that $\rho \geq \int_\mathcal{Q} a(q)\, dq$, hence $\rho = \int_\mathcal{Q} a(q)\, dq$.
	
Next we show convergence in $L^1$.
	Since the abscissa $a$ is continuous on the compact set $\mathcal{Q}$, by the Stone-Weierstra{\ss} theorem \cite[\S 16.4.3]{zorich} it holds that for every $\varepsilon > 0$ there exists a polynomial $h_\varepsilon \in \mathbb{R}[q]$ such that
	\begin{equation*}
		\sup_{q \in \mathcal{Q}} |h_\varepsilon(q) - a(q)| < \frac{\varepsilon}{2}.
	\end{equation*}	
	Hence, the polynomial $v_\varepsilon := h_\varepsilon + \varepsilon$ satisfies $v_\varepsilon - a > 0$ on $\mathcal{Q}$ and we have $v_\varepsilon(q) - x > 0$ on $\mathcal{Z}$. Since the corresponding quadratic module is archimedean (see Remark \ref{qm}), by Putinar's Positivstellensatz \cite[Theorem 2.5]{lasserre} there exist $\sigma_0^\varepsilon, \sigma_{j}^\varepsilon \in \Sigma[q,x,y],\ \tau_\Re^\varepsilon, \tau_\Im^\varepsilon \in \mathbb{R}[q,x,y]$ such that for all $(q,x,y) \in \mathbb{R}^n \times \mathbb{R}^2$ we can write
	\begin{align*}
		v_\varepsilon(q) - x & =  \sigma_0^\varepsilon(q,x,y) + \sum_{j=1}^n \sigma_{j}^\varepsilon(q,x,y)(1-q_j^2)\\ 
		& \hspace{3cm} + \tau_\Re^\varepsilon(q,x,y)p_\Re(q,x,y) + \tau_\Im^\varepsilon(q,x,y)p_\Im(q,x,y).
	\end{align*}	
	Therefore, for $d \geq d_\varepsilon := \lceil \frac{\deg v_\varepsilon}{2} \rceil$ the tuple $(v_\varepsilon, \sigma_0^\varepsilon, \sigma_{j}^\varepsilon, \tau_\Re^\varepsilon, \tau_\Im^\varepsilon)$ is a feasible solution for SDP \eqref{vd} satisfying
	\begin{equation*}
		0 \leq \int_\mathcal{Q} (v_\varepsilon (q) - a(q))\, dq \leq \frac{3\varepsilon}{2} \int_\mathcal{Q} dq.
	\end{equation*}	
	Together with $\int_\mathcal{Q} a(q)\, dq = \rho \leq \rho_d$ which is due to the first part of the proof and $\rho_d$ being a strengthening of $\rho$, it follows that whenever $d \geq d_\varepsilon$ it holds that
	\begin{equation*}
		0 \leq \rho_d - \int_\mathcal{Q} a(q)\, dq \leq \int_\mathcal{Q} (v_\varepsilon (q) - a(q))\, dq \leq \frac{3\varepsilon}{2} \int_\mathcal{Q} dq.
	\end{equation*}	
	As $\varepsilon > 0$ was arbitrary, we obtain $\lim_{d \to \infty} \rho_d = \int_\mathcal{Q} a(q)\, dq$ and since $a \leq v_d$ for all $d$, this is the same as convergence in $L^1$: 
	\begin{multline*}
		0 \leq \lim _{d \to \infty} \|v_d - a\|_1 = \lim_{d \to \infty} \int_\mathcal{Q} |v_d (q) - a(q)|\, dq\\ = \lim_{d \to \infty} \int_\mathcal{Q} (v_d (q) - a(q))\, dq \leq \lim_{d \to \infty} \left(\rho_d + \frac{1}{d}\right) - \int_\mathcal{Q} a(q)\, dq = 0.
	\end{multline*}	
\end{proof}

For linear systems, a polynomial is called stable if all its roots lie in the open left part of the complex plane, i.e. if its abscissa is negative. Hence for a polynomial with parameterized coefficients, as we consider in this paper, the stability region is the set of parameters for which the abscissa is negative, in our notation
\begin{equation*}
	\{q \in \mathcal{Q} : a(q) < 0\}.
\end{equation*}
The following statement on polynomial inner approximations of the  zero sublevel set of the abscissa function follows immediately
from the $L^1$ convergence result of Theorem \ref{converge}, see also \cite{innerpmi}.

\begin{cor}\label{corconverge}
Let $v_d \in {\mathbb R}[q]_{2d}$ denote, as in Theorem \ref{converge}, a near optimal solution for SDP \eqref{vd}. Then
$\{q \in \mathcal{Q} : v_d(q) < 0\} \subset \{q \in \mathcal{Q} : a(q) < 0\}$ and $\lim_{d\to\infty} \mathrm{vol}\:\{q \in \mathcal{Q} :
v_d(q) < 0\} = \mathrm{vol}\:\{q \in \mathcal{Q} : a(q) <0 \}$.
\end{cor}

\subsection{Examples}

As stated in Corollary \ref{corconverge}, while approximating the abscissa function from above we also get an inner approximation of the stability region. The authors of \cite{innerpmi} surveyed a different approach. They described the stability region via the eigenvalues of the Hermite matrix of the polynomial and approximated it using an SDP hierarchy. In the following examples we compare the two different methods and highlight the specific advantages of
our abscissa approximation.

In this section and in the remainder of the paper,
all examples are modelled by Yalmip and solved by Mosek 7 under the Matlab environment,
unless indicated otherwise.

\begin{figure}[h]
\centerline{\includegraphics[width=0.5\textwidth]{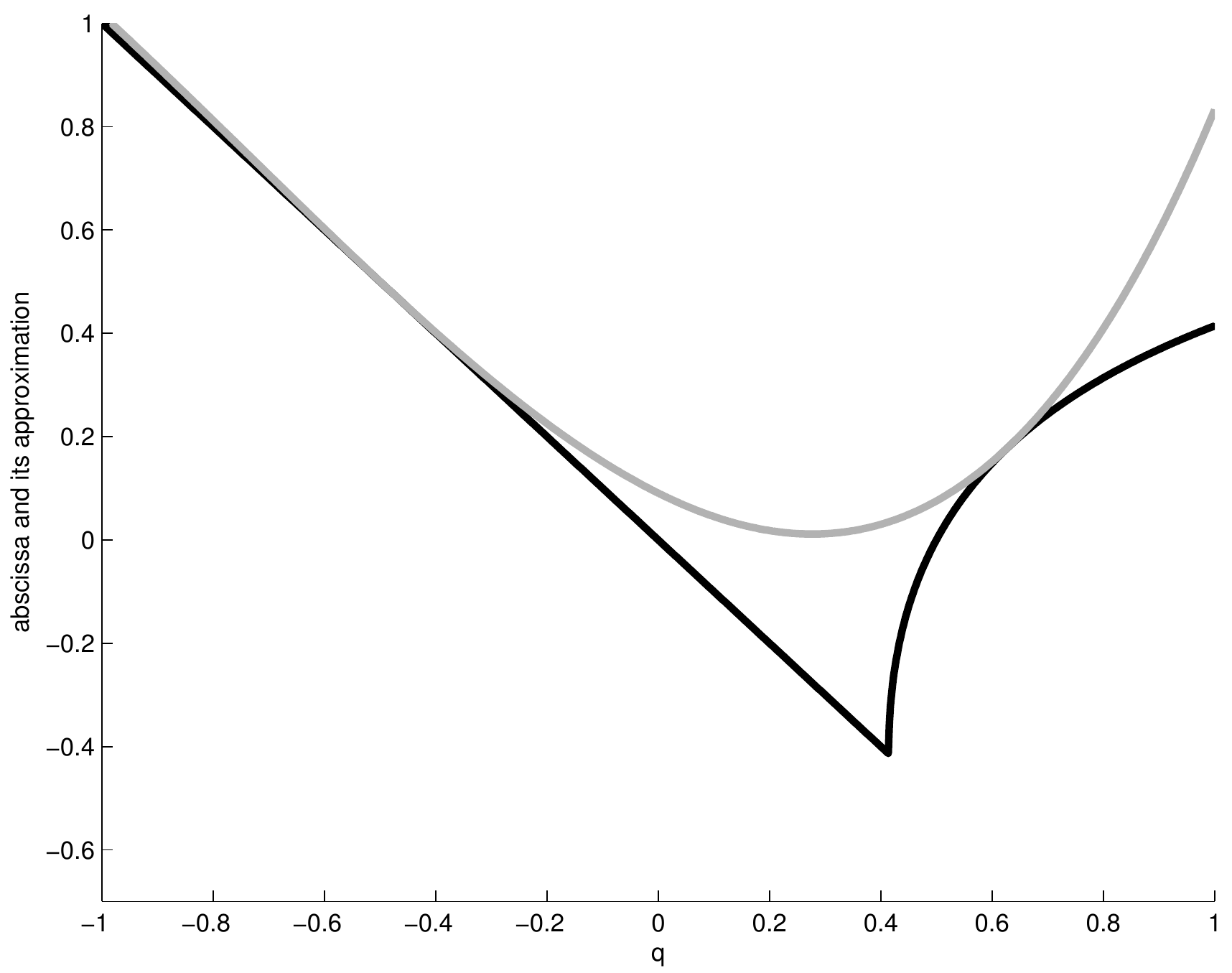}\includegraphics[width=0.5\textwidth]{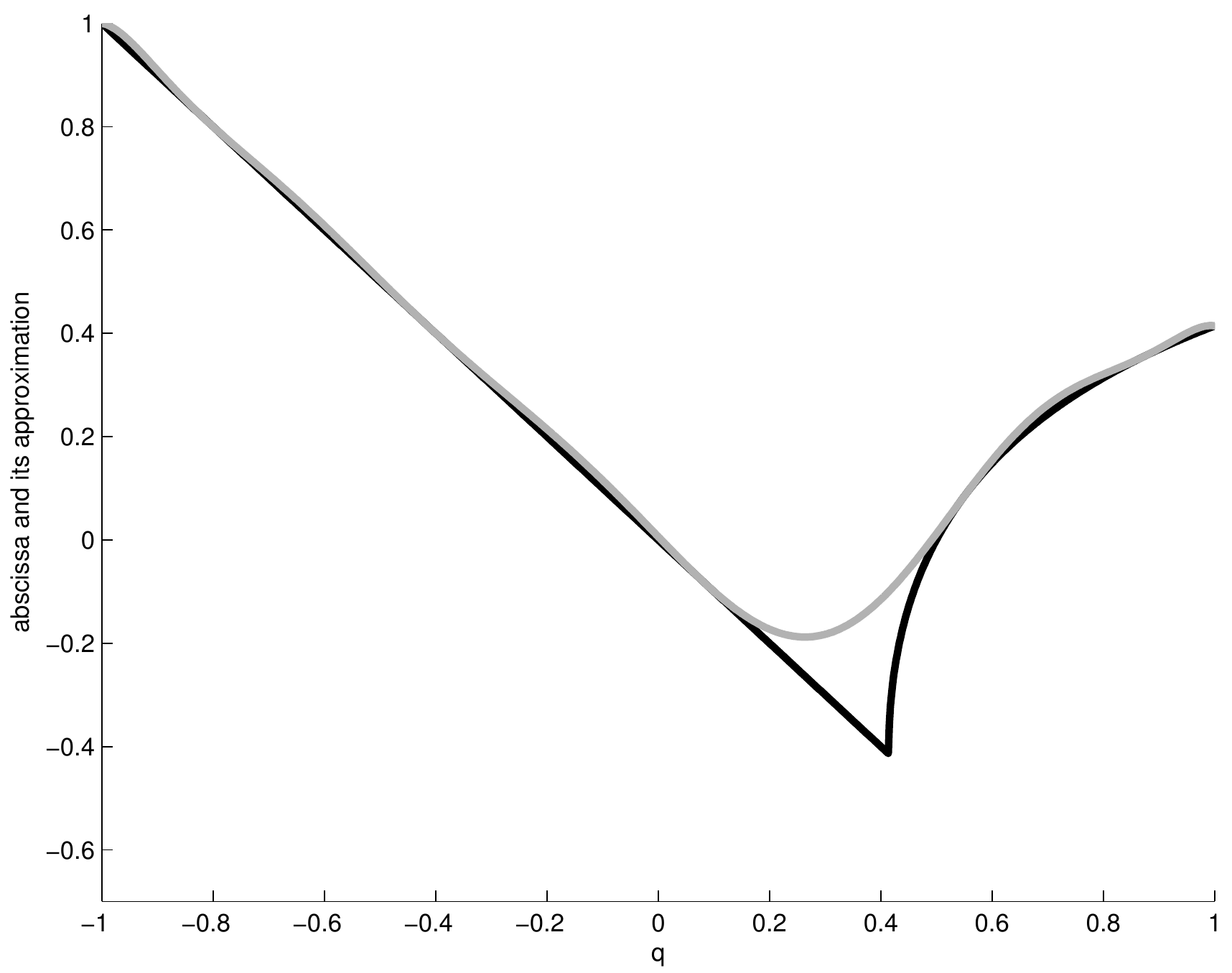}}
\caption{Abscissa (black) and its polynomial upper approximations of degree 4 (left, gray) and 10 (right, gray) for Example \ref{expl1}.
The quality of the approximation deterioriates near the minimum, where the abscissa
is not Lipschitz.}\label{fig-expl1}
\end{figure}

\begin{expl}[The damped oscillator \cite{cross}]\label{expl1} Consider the second degree polynomial depending on
$n=1$ parameter $q \in \mathcal{Q}=[-1,1]$:
\begin{equation*}
	p:\ s \mapsto p(q,s) = s^2 + 2qs + 1-2q.
\end{equation*}
Then $\mathcal{Z} = \{(q,x,y) \in [-1,1] \times \mathbb{R}^2 : x^2-y^2+2qx+1-2q = 2xy+2qy = 0\}$ and the corresponding hierarchy of SDP problems \eqref{vd} reads
	\begin{align*}
		\rho_d = & \inf_{v_d,\sigma_0,\sigma_1,\tau_\Re,\tau_\Im} \int_{-1}^{1} v_d(q)\, dq\\
		& \text{s.t.}\ v_d(q) - x = \sigma_0(q,x,y) + \sigma_1(q,x,y)(1-q^2)\\
		&\hspace{2,5cm} + \tau_\Re(q,x,y)(x^2-y^2+2qx+1-2q) + \tau_\Im(q,x,y)(2xy+2qy)
	\end{align*}
for all $(q,x,y) \in \mathbb{R}^3$ and with $v_d \in \mathbb{R}[q]_{2d}$, $\sigma_0 \in \Sigma[q,x,y]_{2d},\ \sigma_1 \in \Sigma[q,x,y]_{2d-2}$ and $\tau_\Re, \tau_\Im \in \mathbb{R}[q,x,y]_{2d-2}$. Apart from that, we only need the moments of the Lebesgue measure on $[-1,1]$ for a successful implementation. These are readily given by
\begin{equation*}
	z_\alpha = \int_{-1}^1 q^\alpha\, dq = \frac{1-(-1)^{\alpha+1}}{\alpha+1},
\end{equation*}
meaning that $\int_{-1}^{1} v_d(q)\, dq = \sum_{\alpha=1}^d {v_d}_{\alpha} z_\alpha$ with ${v_d}_{\alpha}$ denoting the coefficient of the monomial $q^\alpha$ of $v_d$. See Figure \ref{fig-expl1} for the graphs of the degree 4 and 10 polynomial upper approximations of the abscissa.

For the Hermite approximation we compute the Hermite matrix $H$ of $p$ (see \cite{stability} for details)
\begin{equation*}
	H(q) = \begin{pmatrix} 4q-8q^2 & 0\\ 0 & 4q \end{pmatrix}
\end{equation*}
and write the hierarchy of optimization problems as presented in \cite{innerpmi}:
	\begin{align*}
		&\max_{g_d,\sigma_0,\sigma_1,\tau} \int_{-1}^{1} g_d(q)\, dq\\
		&\text{s.t.}\ u^TH(q)u - g_d(q) = \sigma_0(q,u) + \sigma_1(q,u)(1-q^2) + \tau(q,u)(1-u^Tu)
	\end{align*}	
	for all $(q,u) \in [-1,1] \times \mathbb{R}^2$ and with $g_d \in \mathbb{R}[q]_{2d}$, $\sigma_0 \in \Sigma[q,u]_{2d},\ \sigma_1 \in \Sigma[q,u]_{2d-2}$ and $\tau \in \mathbb{R}[q,u]_{2d-2}$. Already for $d=6$ we observe a close match between the
genuine stability region, which is $\{q \in [-1,1] : a(q) < 0\} = (0,\tfrac{1}{2})$, the Hermite inner approximation 
$\{q \in [-1,1] : -g_6(q) < 0\}$, and the polynomial upper approximation $\{q \in [-1,1] : v_{10}(d) < 0\}$.
These three intervals are visually indistinguishable, so we do not represent them graphically.
\end{expl}

\begin{figure}[h]
\centerline{\includegraphics[width=0.5\textwidth]{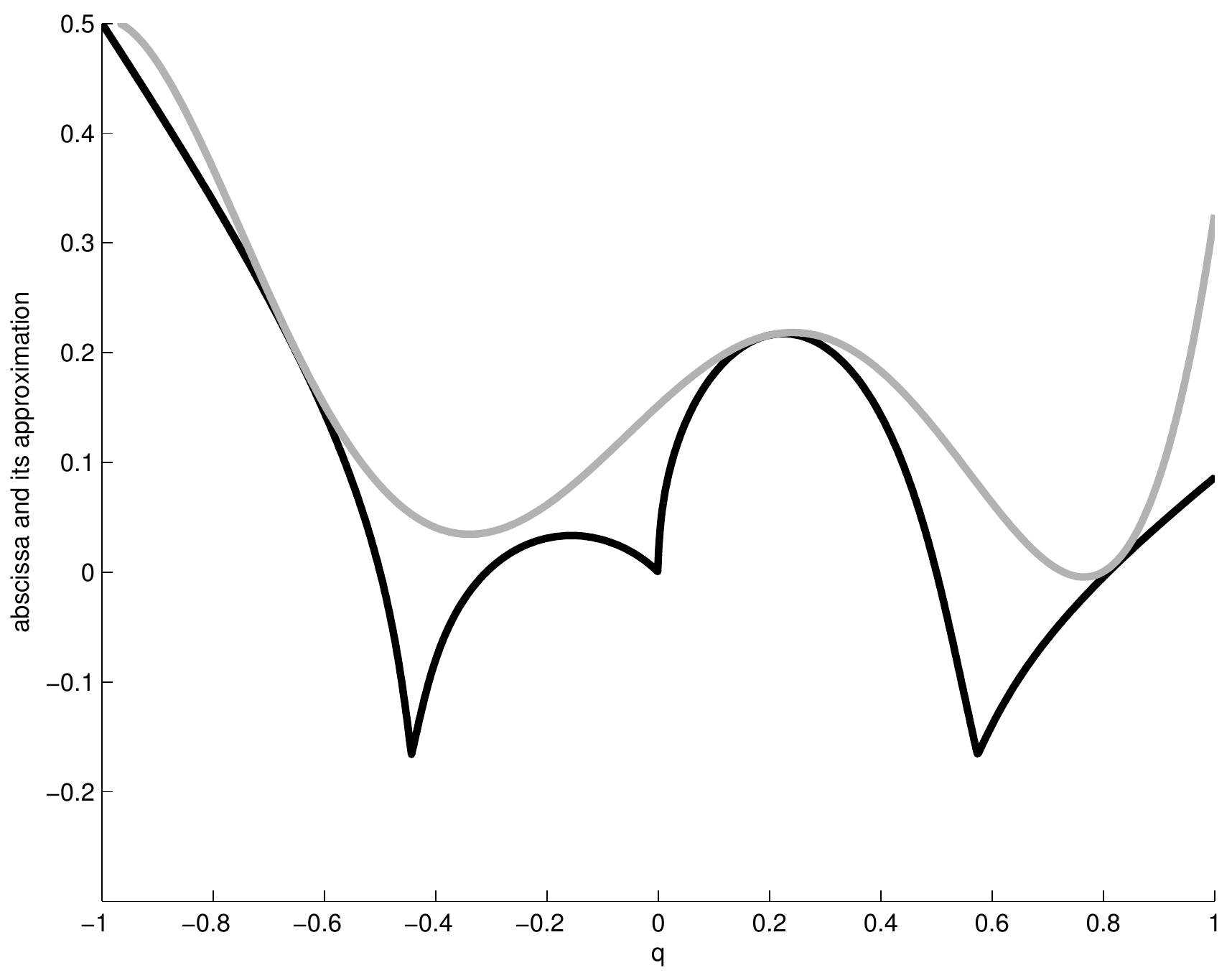}\includegraphics[width=0.5\textwidth]{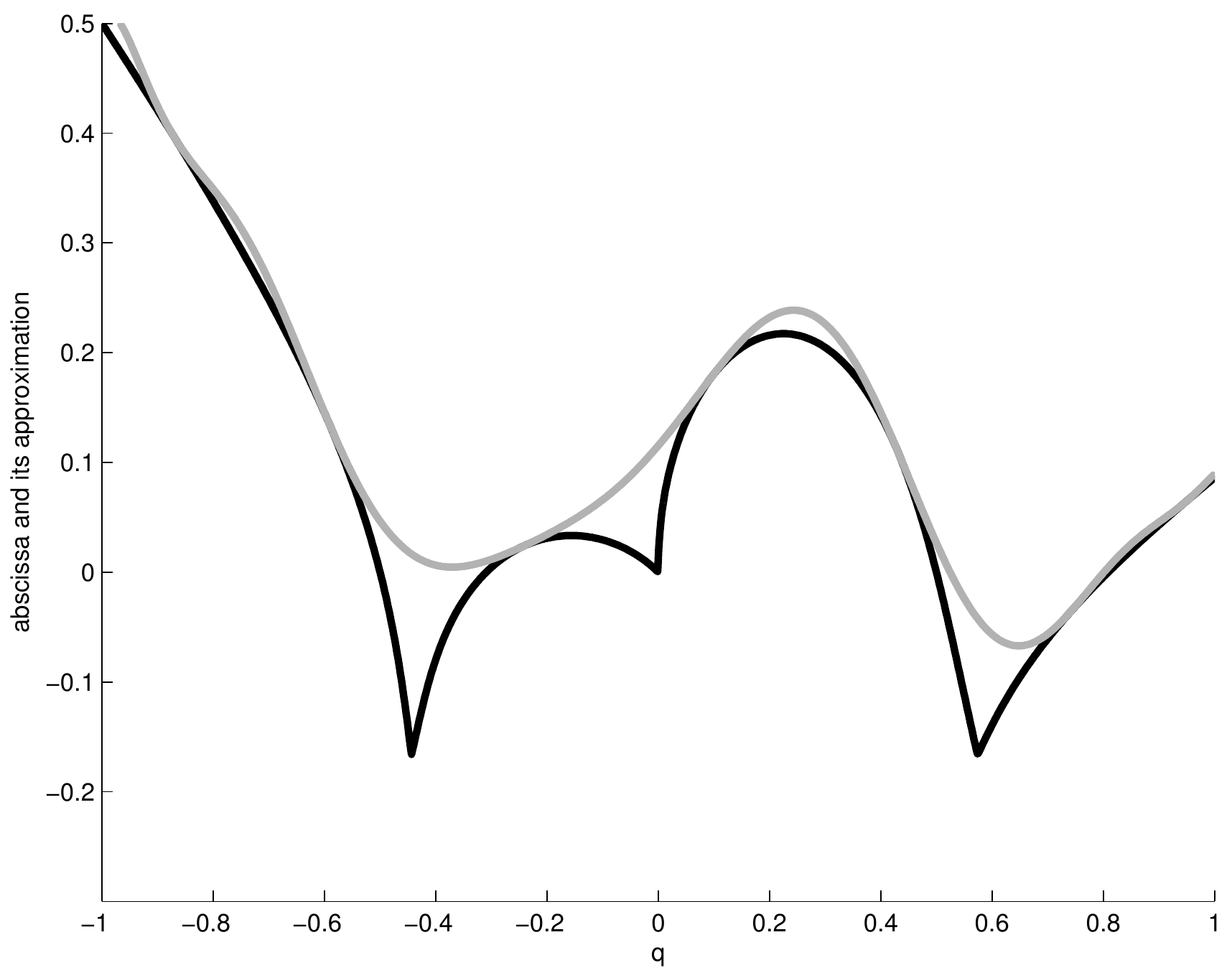}}
\caption{Abscissa (black) and its polynomial upper approximations of degree 6 (gray, left) and 12 (gray, right)
for Example \ref{expl2}.
The quality of the approximation deterioriates near the points of non-differentiability of the abscissa.}\label{fig-expl2}
\end{figure}

\begin{expl}\label{expl2}
Consider the polynomial
	\begin{equation*}
	p:\ s \mapsto p(q,s) = s^3 + \tfrac{1}{2}s^2 + q^2s + (q-\tfrac{1}{2})q(q+\tfrac{1}{2})
\end{equation*}
for $q \in \mathcal{Q} = [-1,1]$.
The abscissa function $a(q)$ of $p$ is not differentiable at three points and therefore it is rather hard to approximate in their
neighborhoods. In Figure \ref{fig-expl2} we see the abscissa and its polynomial upper approximations of degrees 6 and 12.
Comparing the genuine stability region
$\{q \in [-1,1] : a(q) < 0\}$, the polynomial inner approximation $\{q \in [-1,1] : v_{12}(q) < 0\}$ and the Hermite inner approximation
$\{q \in [-1,1] : -g_{10}(q)< 0\}$, we observe, maybe surprisingly, that the approximations are very similar and miss the same parts of the stability region. These are not reproduced graphically.
\end{expl}

\begin{rem}
Evidently, the approach via the Hermite matrix does not tell us anything about the abscissa function itself besides from where it is negative. As an illustration consider a polynomial of the form $p(q,s) = s^2 + p_0(q)$ for $n=1$. Then $p(q,\cdot)$ has either $0$ as a multiple root, two real roots (of which one is positive) or only imaginary roots, i.e. the stability region of $p$ is empty and its Hermite matrix $H(q)$ is zero. Therefore the eigenvalues and their approximation $g_d$ are also zero for every $d$. In contrast, the upper abscissa approximation $v_d$ gives a suitable approximation for the abscissa function.

On the other hand, practical experiments (not reported here) reveal that computing the abscissa approximation is typically more challenging numerically than computing the Hermite approximation. For instance, computing the upper abscissa approximation may fail for polynomials with large coefficients, while the Hermite approximation keeps providing a proper inner approximation of the stability region.
\end{rem}

\begin{figure}[h]
\centerline{\includegraphics[width=0.5\textwidth]{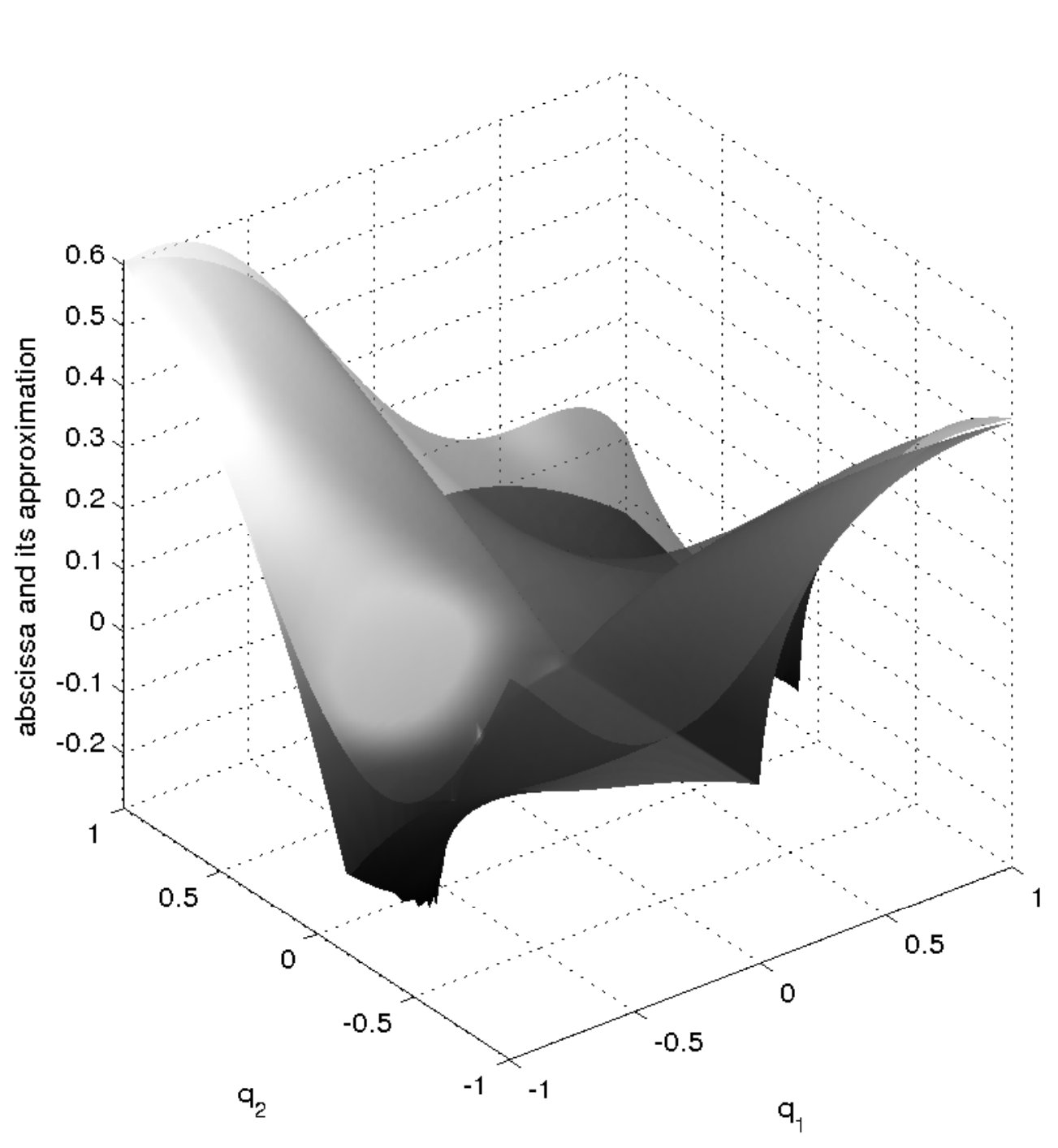}\includegraphics[width=0.5\textwidth]{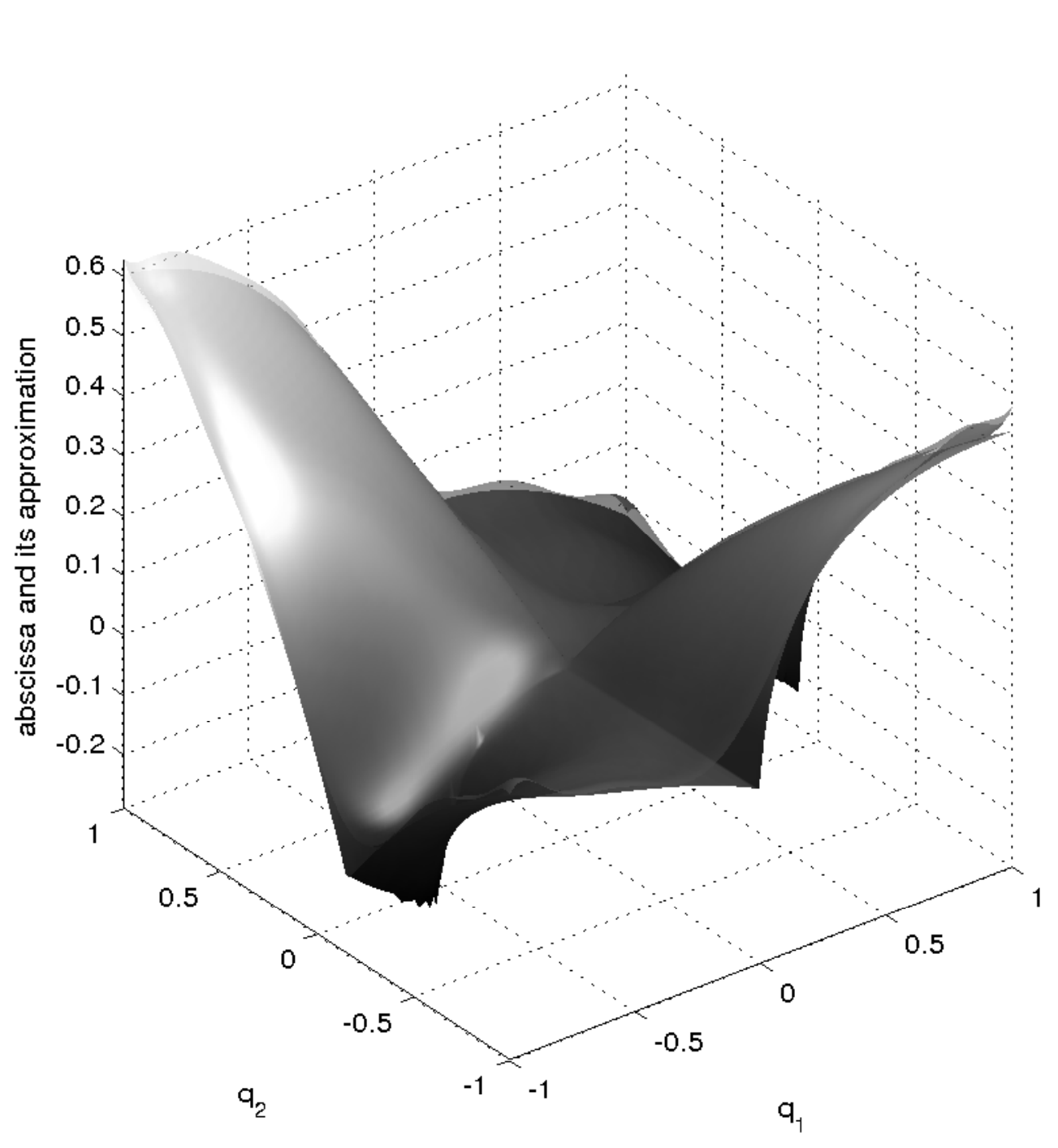}}
\caption{Abscissa (dark, below) and its polynomial upper approximations of degrees 6 (left, transparent) and
10 (right, transparent) for Example \ref{expl3}.
We observe that the approximation deteriorates near the regions of non-differentiability of the abscissa.}\label{fig-expl3}
\end{figure}

\begin{figure}[h]
\centerline{\includegraphics[width=0.5\textwidth,height=0.3\textheight]{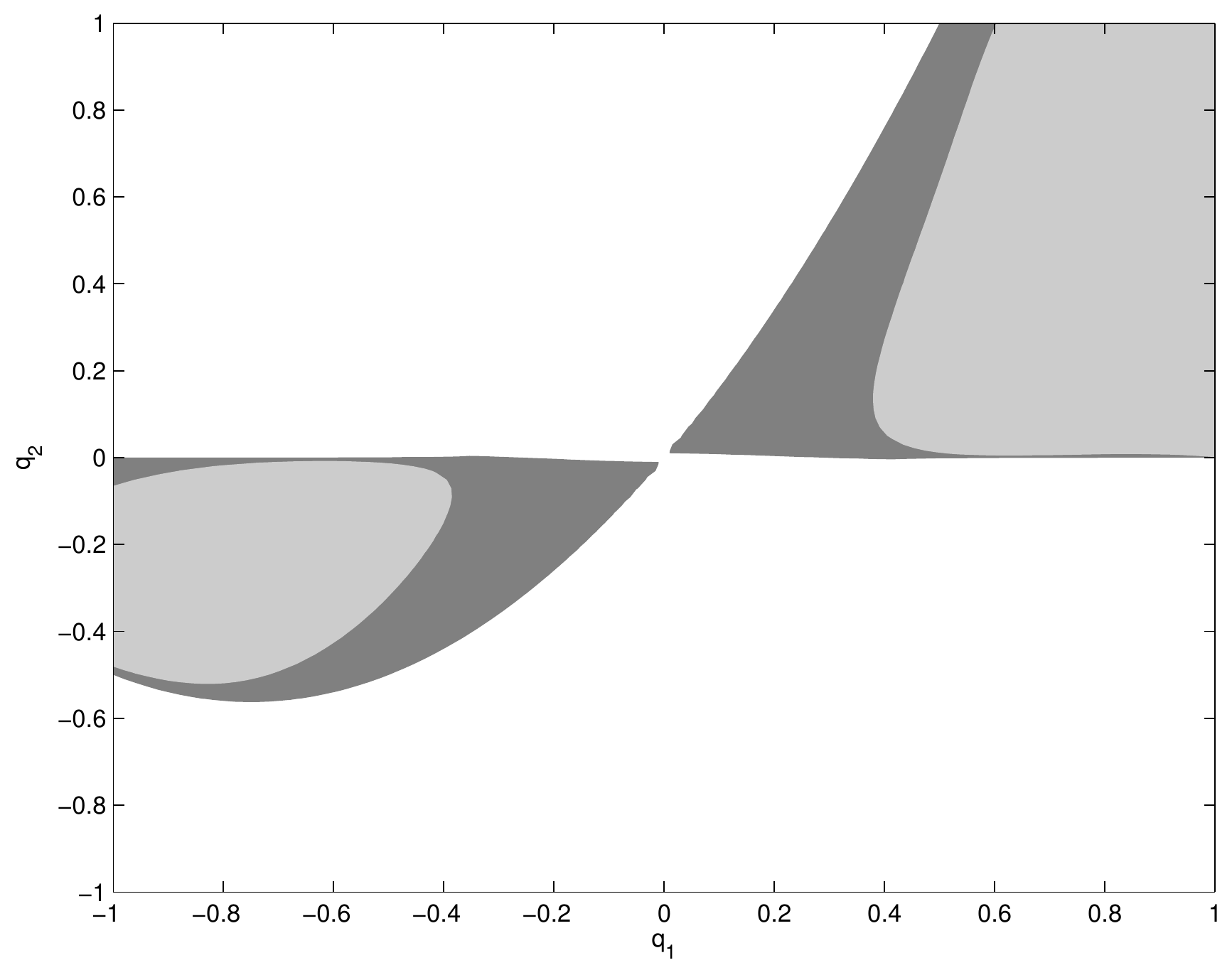}
\includegraphics[width=0.5\textwidth,height=0.3\textheight]{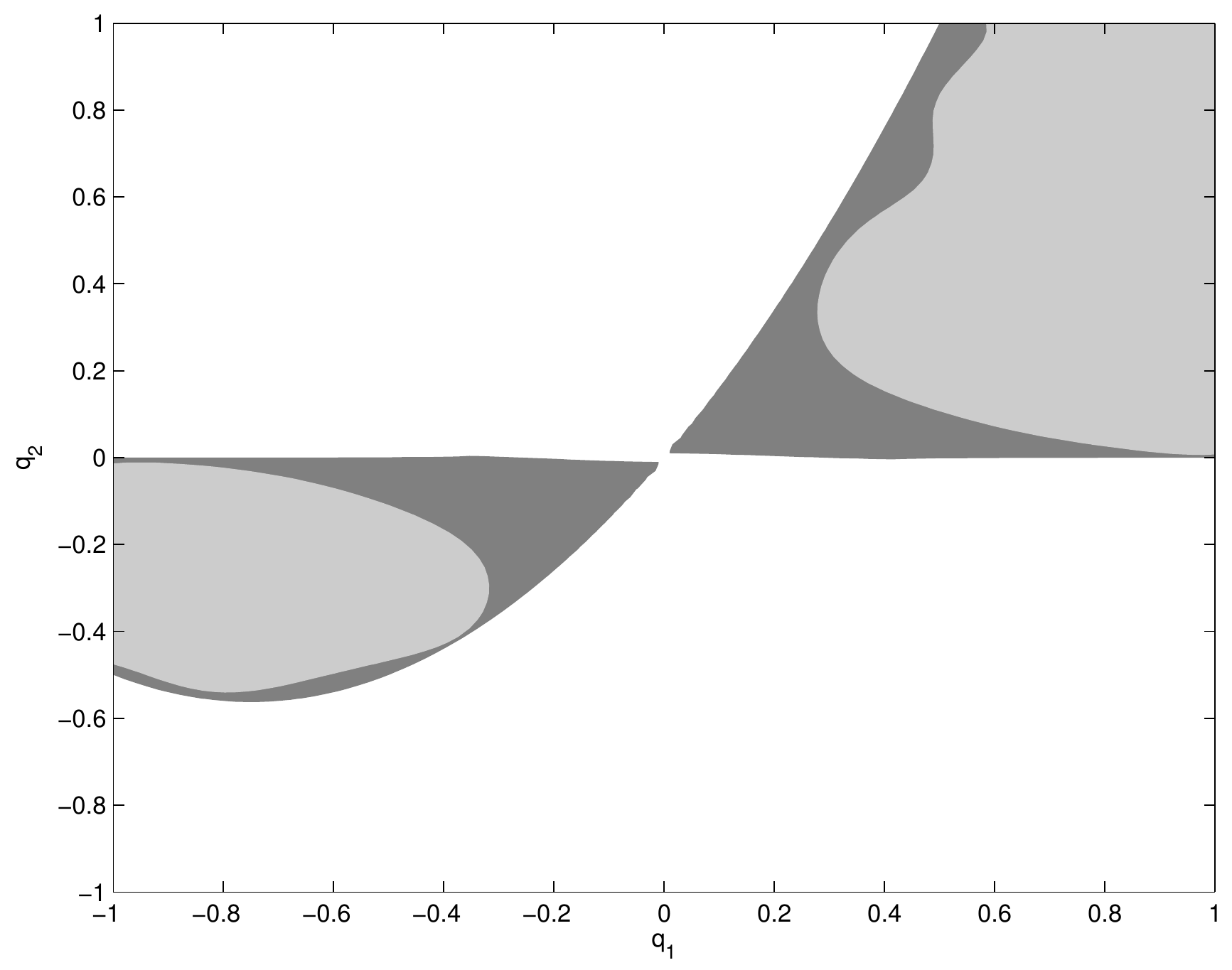}}
\caption{Stabilizability region (dark gray region) and its inner approximations with degree 8 Hermite (light gray region, left) and degree 10 upper polynomial approximation (light gray region, right).}\label{fig-expl3bis}
\end{figure}

\begin{expl}\label{expl3}
Consider the polynomial
	\begin{equation*}
		p:\ s \mapsto p(q,s) = s^3 + (q_1+\tfrac{3}{2}) s^2 + q_1^2s + q_1q_2
	\end{equation*}
depending on $n=2$ parameters $q \in \mathcal{Q}=[-1,1]^2$.
	Then $\mathcal{Z} = \{(q,x,y) \in [-1,1]^2 \times \mathbb{R}^2 : x^3-3xy^2+(q_1+\tfrac{3}{2})x^2-(q_1+\tfrac{3}{2})y^2+q_1^2x+q_1q_2 = -y^3+3x^2y+2(q_1+\tfrac{3}{2})xy+q_1^2y = 0\}$.
	In Figure \ref{fig-expl3} we represent the graphs of the abscissa $a$ and its polynomial approximations $v_6$ and $v_{10}$.
In Figure \ref{fig-expl3bis} we represent the stabilizability region, i.e. the zero sublevel set of the abscissa
$\{q \in [-1,1]^2 : a(q) < 0\}$ (dark gray region),
the degree 8 Hermite sublevel set $\{q \in [-1,1]^2 : -g_8(q) < 0\}$ (light gray region, left) and
the degree 10 polynomial sublevel set $\{q \in [-1,1]^2 : v_{10}(q) < 0\}$ (light gray region, right).
\end{expl}

\begin{rem}
	In the examples we always chose lower degrees for the Hermite approximation than for the upper abscissa approximation. The Hermite approximation converges relatively fast making it unnecessary to consider higher degrees, especially since they require much more time.
On the contrary, the upper abscissa approximation usually needs higher degrees to provide a useful approximation, but it is faster to compute.
\end{rem}

\section{Lower abscissa approximation}

At first thought, finding a lower approximation for the abscissa map might sound like a straightforward task, since one is tempted to just solve the analogue of LP \eqref{v}:
\begin{align}
&\sup_{w \in \mathscr{C}(\mathcal{Q})}  \int_\mathcal{Q} w(q)\, dq \label{vlower}\\
&\text{s.t.}\quad x - w(q) \geq 0 \text{ for all } (q,x,y) \in \mathcal{Z}. \notag
\end{align}

This, indeed, gives a valid lower bound on the abscissa function, however in general a very bad one since it is not approximating the abscissa but the minimal real part of the roots of $p$. To understand the reason we recall that
\begin{equation*}
	\mathcal{Z} = \{(q,x,y) \in \mathbb{R}^n \times \mathbb{R}^2 : q \in \mathcal{Q},\: p_\Re(q,x,y) = p_\Im(q,x,y) = 0\}
\end{equation*}
and therefore this set contains all roots of $p$ and not only those with maximal real part. 

\begin{expl}\label{expl-lower}
On the left of Figure \ref{fig-lower} we show the degree 12 solution to the SDP hierarchy corresponding
to LP \eqref{vlower} for the polynomial $p(q,s) = s^2 + 2qs + 1-2q$ of Example \ref{expl1},
which gives a tight lower approximation to the abscissa only in the left part of the domain, 
corresponding to a pair of complex conjugate roots. We observe that the SDP solver Mosek does not return
a correct answer for this particular problem, and we had to use the SDP solver SeDuMi instead in this case.
On the right of Figure \ref{fig-lower} we show 
 the degree 12 solution to the SDP hierarchy corresponding to LP \eqref{vlower} for
the polynomial $p(q,s) = s^3 + \tfrac{1}{2}s^2 + q^2s + (q-\tfrac{1}{2})q(q+\tfrac{1}{2})$ of Example \ref{expl2}.
The lower approximation is nowhere tight, due to the presence of roots with real parts smaller than the abscissa.
\end{expl}

\begin{figure}[h]
\centerline{\includegraphics[width=0.5\textwidth]{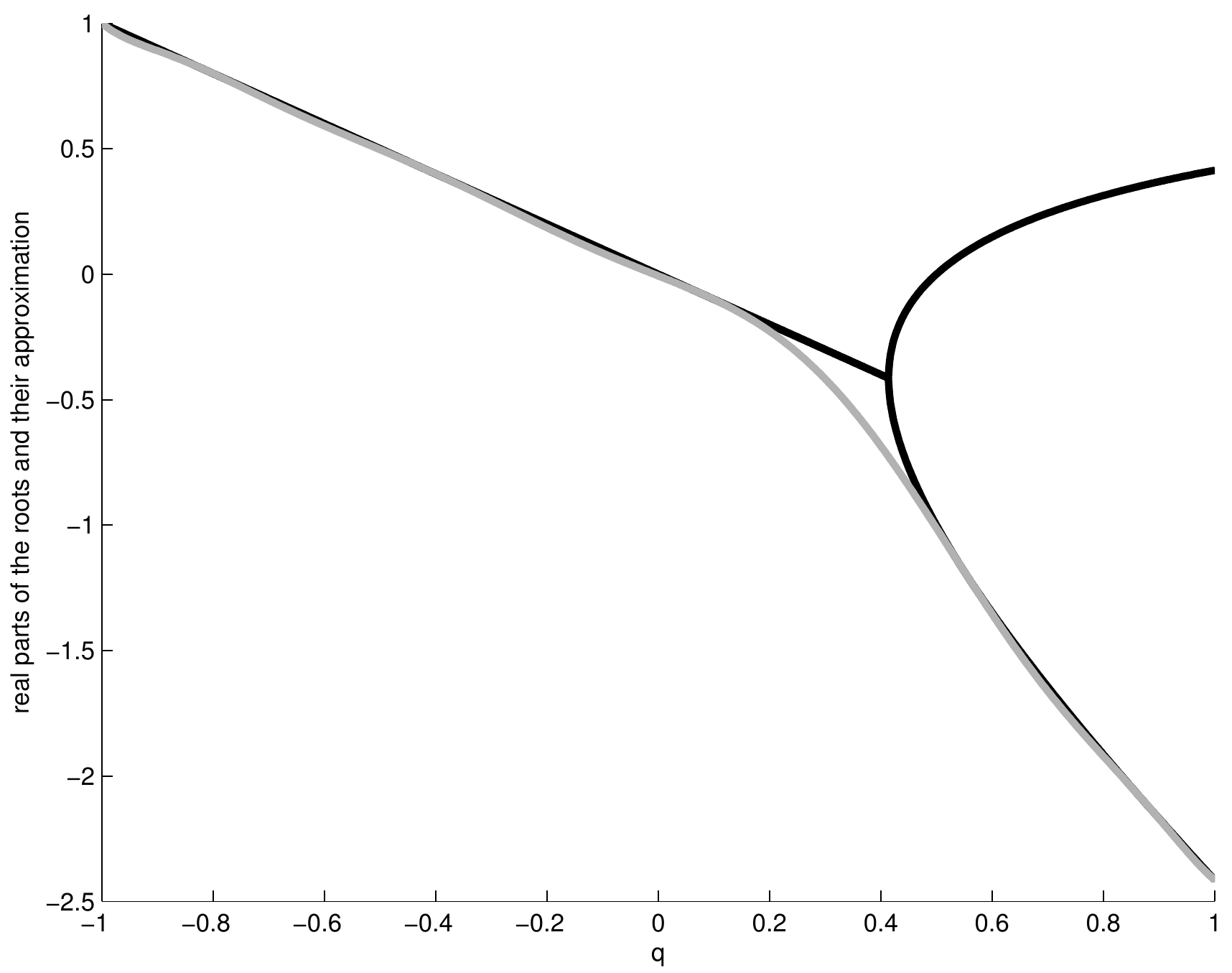}
\includegraphics[width=0.5\textwidth]{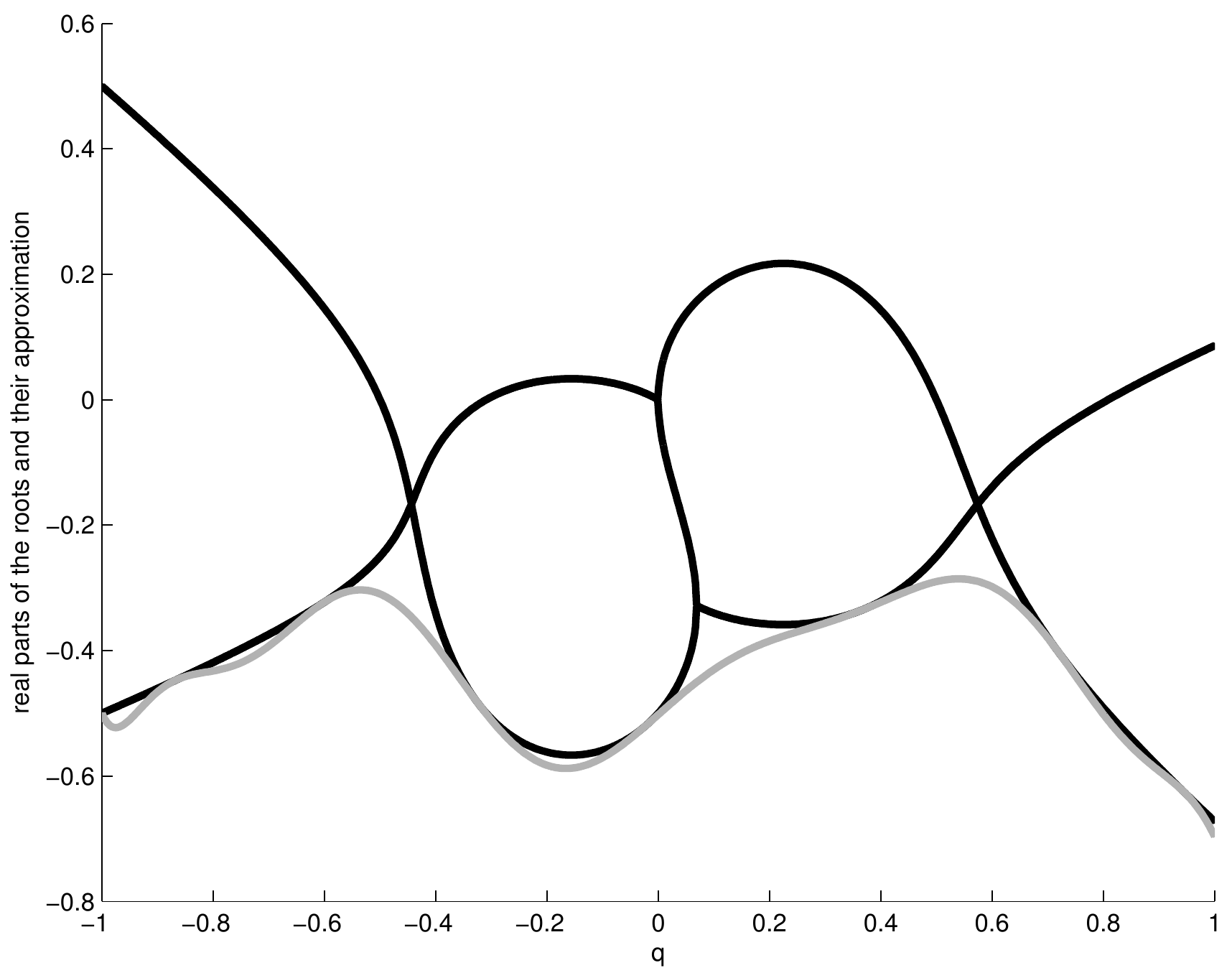}}
\caption{Real parts of the roots (black) and degree 12 polynomial lower approximations (gray) for the
second degree polynomial (left) and third degree polynomial (right) of Example \ref{expl-lower}.}\label{fig-lower}
\end{figure}

To find a tighter approximation for the abscissa map from below we pursue two different approaches:
\begin{itemize}
\item First, we reformulate the set $\mathcal{Z}$ with the help of elementary symmetric functions, in order to have access to the roots directly. This is a very neat way with options for variation, such as approximating the second largest real part of the roots from above or below, but it also includes many additional variables and it is therefore not very efficient when implemented. However, it can be useful for small problems.
\item Second, we restrict LP \eqref{vlower} further using the Gau{\ss}-Lucas theorem, i.e. instead of $\mathcal{Z}$ we use a subset of $\mathcal{Z}$ which contains only the roots with the abscissa as its real parts. This approach is much more complicated, relies on assumptions and one needs to solve two optimization problems in order to get the lower approximation. Nevertheless, the implementation is much faster and it can also be used for bigger problems.
\end{itemize}

\subsection{Lower approximation via elementary symmetric functions}\label{esf}

\subsubsection{Problem formulation}

Let us derive another description of the set of roots of $p$ which allows us to pick single roots according to the size of their real part. For this purpose let us recall the definition of our polynomial:
\begin{equation*}
	p:\ s \mapsto p(q,s) := \sum_{k=0}^m p_k(q)s^k\quad \text{with} \quad p_m(q) \equiv 1.
\end{equation*}
Following the notation of the previous chapters, we denote the roots of $p(q,\cdot)$ by $s_k(q),\ k=1,\dotsc,m$ and split them up into their real and imaginary parts, $s_k(q) = x_k(q) + iy_k(q)$ with $x_k(q), y_k(q) \in \mathbb{R}$. To simplify notations we omit the dependence on $q$ whenever it is clear and write only $s_k,\ x_k$ and $y_k$.

Now we write the coefficients of the polynomial as elementary symmetric functions of its roots:
\begin{equation*}
 p_{m-k}(q) =(-1)^k \sum_{1 \leq l_1<l_2<\cdots<l_k \leq m} s_{l_1}s_{l_2} \cdots s_{l_k},\quad k = 1,\dotsc\,m.
\end{equation*}
This allows us to define the set of roots of $p$ in the following way, where we can order the roots arbitrarily:
\begin{multline*}
	\mathcal{Z}'_o := \{(q,x_1,\dotsc,x_m,y_1,\dotsc,y_m) \in \mathcal{Q} \times \mathbb{R}^m \times \mathbb{R}^m  : x_k \leq x_m,\: k=1,\ldots,m-1,\\ 
	p_{m-k}(q) = (-1)^k \sum_{1 \leq l_1<l_2<\cdots<l_k \leq m} s_{l_1}s_{l_2} \cdots s_{l_k},\quad k = 1,\dotsc\,m\}.
\end{multline*}
To avoid complex variables $s_{l_k}$ in the description of the set, we replace them by $s_{l_k} = x_{l_k} + i y_{l_k}$ and split up the sum $\sum_{1 \leq l_1<\cdots<l_k \leq m} s_{l_1}s_{l_2} \cdots s_{l_k}$ in its real and imaginary parts. The latter would be zero, since all $p_{m-k}(q)$ are real. In the sequel we omit this procedure, since it would only complicate the notations.

For illustrative reasons let us fix $q$ for a moment. Then the set $\mathcal{Z}'_o$ contains only one element $(q,x_1,\dotsc,x_m,y_1,\dotsc,y_m)$. For this it holds that $x_m = a(q)$ and the points $(q,x_k,y_k),\ k=1,\dotsc,m$, are exactly the elements of $\mathcal{Z}$.

\begin{rem}
	One could order the roots further by adding more conditions, like for example $x_k \leq x_{m-1},\ k=1,\cdots,m-2$. Then one could also access the root with the second largest real part. Of course, this would imply another $m-2$ constraints in an implementation and therefore this would slow down further the solving process.
\end{rem}

In theory, $m$ variables suffice to characterize the roots of a real polynomial via the elementary symmetric functions, but since we need all variables $x_k$ explicitly in order to identify the maximal one, we can only eliminate $\left\lfloor \tfrac{m}{2} \right\rfloor := \max \{r \in \mathbb{Z} \mid r \leq \tfrac{m}{2}\}$ variables. We set
\begin{equation*}
	y_{k-1} = -y_k,\ \begin{cases}
		k = 2,\dotsc,m &\text{if $m$ is even}\\ k = 2,\dotsc,m-3\ \text{and } y_{m-2} = -y_{m-1}-y_m &\text{if $m$ is odd,}
	\end{cases}
\end{equation*}
meaning we decide which roots will be pairs in case they are complex. Note that it is necessary to keep $y_m$, since we defined $x_m$ as the abscissa and we do not know whether $s_m$ is real or not. In fact $s_m(q)$ can be real for some $q$ and complex for others.

\begin{rem}
	Even though we know for $m$ odd that one root must be real, we cannot eliminate $\left\lceil \tfrac{m}{2} \right\rceil$ variables, since it might happen that $s_m(q)$ is the single real root for some $q$ while it is complex for other $q$.
\end{rem}

Now we can write the set of roots with less variables and less constraints. As above we keep the variables $s_k$ in the description of the set for readability reasons, but remark that with the reduced amount of $y$ variables the constraints $0 = \Im(\sum_{1 \leq l_1<\cdots<l_k \leq m} s_{l_1}s_{l_2} \cdots s_{l_k})$ for $k = 1,\dotsc,\left\lfloor \tfrac{m}{2} \right\rfloor$ are superfluous. We have
\begin{align*}
	\mathcal{Z}_o := &\{(q,x_1,\dotsc,x_m,y_2,y_4,\dotsc,y_{2\left\lfloor \tfrac{m}{2} \right\rfloor},y_m) \in \mathcal{Q} \times \mathbb{R}^m \times \mathbb{R}^{\left\lceil \tfrac{m}{2} \right\rceil} : \\
          &\hspace{2cm}x_k \leq x_m,\: k=1,\ldots,m-1,\\
	&\hspace{2cm}p_{m-k}(q) = (-1)^k \sum_{1 \leq l_1<l_2<\cdots<l_k \leq m} s_{l_1}s_{l_2} \cdots s_{l_k},\: k = 1,\dotsc\,m\}.
\end{align*}

\begin{expl}\label{explm3}
	For $m=3$ the set $\mathcal{Z}_o$ is given by
	\begin{align*}
		\mathcal{Z}_o = &\{(q,x_1,x_2,x_3,y_2,y_3) \in \mathcal{Q} \times \mathbb{R}^3 \times \mathbb{R}^2 : x_1 \leq x_3,\ x_2 \leq x_3,\\
		&\hspace{2cm} -p_2(q) =  x_1 + x_2 + x_3,\\
		&\hspace{2cm} p_1(q) = x_1x_2+x_1x_3+x_2x_3 + y_2^2+y_2y_3+y_3^2,\\ 
		&\hspace{2cm} -p_0(q) = x_1x_2x_3 + (-x_1+x_2+x_3)y_2y_3 + x_2y_3^2 + x_3y_2^2,\\
		&\hspace{2cm} 0 = (x_1-x_2)y_2 + (x_1-x_3)y_3,\\
		&\hspace{2cm} 0 = (x_1-x_2)x_3y_2 + (x_1-x_3)x_2y_3 + y_2^2y_3+y_2y_3^2\}.
	\end{align*}
	To clarify the formula also for $m$ even, we write $\mathcal{Z}_o$ down explicitly for $m=4$:
	\begin{align*}
		\mathcal{Z}_o = &\{(q,x_1,x_2,x_3,x_4,y_2,y_4) \in \mathcal{Q} \times \mathbb{R}^4 \times \mathbb{R}^2 : x_1 \leq x_4,\ x_2 \leq x_4,\ x_3 \leq x_4,\\
		&\hspace{1cm} -p_3(q) = x_1 + x_2 + x_3 + x_4,\\
		&\hspace{1cm} p_2(q) = x_1x_2+x_1x_3+x_1x_4+x_2x_3+x_2x_4+x_3x_4 + y_2^2+y_4^2,\\ 
		&\hspace{1cm} -p_1(q) = x_1x_2(x_3+x_4)+(x_1+x_2)x_3x_4 + (x_1+x_2)y_4^2 + (x_3+x_4)y_2^2,\\
		&\hspace{1cm} p_0(q) = x_1x_2x_3x_4+ (x_1-x_2)(x_4-x_3)y_2y_4 + x_1x_2y_4^2+x_3x_4y_2^2 + y_2^2y_4^2,\\
		&\hspace{1cm} 0 = (x_1-x_2)(x_3+x_4)y_2 + (x_1+x_2)(x_3-x_4)y_4,\\
		&\hspace{1cm} 0 = (x_1-x_2)(x_3x_4+y_4^2)y_2 + (x_3-x_4)(x_1x_2+y_2^2)y_4\}.
	\end{align*}	
	Here we have set $y_1 = -y_2$ and $y_3 = -y_4$, so the constraint $0 = \Im(\sum_{1 \leq l_1<\cdots<l_k \leq m} s_{l_1}s_{l_2} \cdots s_{l_k})$ for $k=1$ is obviously superfluous, because it reduces to $0=0$. The second superfluous constraint is the one for $k=2$, that is $0 = (x_1-x_2)y_2 + (x_3-x_4)y_4$, since we have $x_1=x_2$, respectively $x_3=x_4$, in the case $s_2$, respectively $s_4$, is complex.
\end{expl}

Finally, we can reformulate  LP \eqref{vlower} in such a way that it provides a proper approximation of the abscissa function from below:
\begin{align}
	\rho= & \sup_{w \in \mathscr{C}(\mathcal{Q})}  \int_\mathcal{Q} w(q)\, dq \label{tau}\\
	& \text{s.t.}\quad x_m - w(q) \geq 0 \text{ for all } (q,x_1,\dotsc,x_m,y_2,y_4,\dotsc,y_{2\left\lfloor \tfrac{m}{2} \right\rfloor},y_m) \in \mathcal{Z}_o. \notag
\end{align}
With the notation of Section \ref{sec1} its dual LP reads
\begin{align}
	\rho^* = & \inf_{\mu \in \mathscr{M}^+(\mathcal{Z}_o)}  \int_{\mathcal{Z}_o} x_m\ d\mu(q,x_1,\dotsc,x_m,y_2,y_4,\dotsc,y_{2\left\lfloor \tfrac{m}{2} \right\rfloor},y_m) \label{mum}\\
	& \text{s.t.}  \int_{\mathcal{Z}_o} q^\alpha\, d\mu = \int_\mathcal{Q} q^\alpha\, dq, \text{ for all } \alpha \in \mathbb{N}^n. \notag
\end{align}
In analogy with the upper approximation we have no duality gap and the infimum is attained:
\begin{lem}\label{zerod}
The infimum in LP (\ref{mum}) is attained, and there is no duality gap between LP (\ref{tau}) and LP (\ref{mum}), i.e.
$\rho=\rho^*$.
\end{lem}
Since $\mathcal{Z}_o$ is compact, the proof is identical to that of Lemma \ref{zero}. 

\begin{rem}\label{attainedd}
	For the same reasons as for the upper approximation \eqref{v}, the supremum in \eqref{tau} is not attained for $\mathscr{C}(\mathcal{Q})$ or $\mathbb{R}[q]$, but it is attained for $\mathbb{R}[q]_d$ with $d$ finite. See Remark \ref{attained} with $M := \min_{q \in \mathcal{Q}} a(q) - N$ for an $N \in \mathbb{N}$ sufficiently large, and $R := \int_{\mathcal{Q}} \left(a(q)-M\right)\, dq$.
\end{rem}

\subsubsection{SDP hierarchy}

Let $d_0 \in \mathbb{N}$ be sufficiently large. Then for $d \in \mathbb{N},\ d \geq d_0$ the corresponding hierarchy of SDP problems reads
\begin{align}
		\rho_d = & \sup_{w_d,\sigma_0,\sigma_k,\sigma_{x_l},\tau_{\Re,k},\tau_{\Im,k}} \int_\mathcal{Q} w_d(q)\, dq\label{taud}\\
		& \text{s.t.}\ x_m - w_d(q) = \sigma_0 + \sum_{j=1}^n \sigma_{j}(1-q_j^2) + \sum_{k=1}^{m-1} \sigma_{x_k}(x_m-x_k)\notag\\
		&\hspace{2.4cm} + \sum_{k=1}^m \tau_{\Re,k}\left((-1)^k p_{m-k}(q)-\Re\left(\sum_{1 \leq l_1<l_2<\cdots<l_k \leq m} s_{l_1}s_{l_2} \cdots s_{l_k}\right)\right)\notag\\
		&\hspace{2.4cm} + \sum_{k=\left\lfloor \tfrac{m}{2} \right\rfloor}^m \tau_{\Im,k}\Im\left(\sum_{1 \leq l_1<l_2<\cdots<l_k \leq m} s_{l_1}s_{l_2} \cdots s_{l_k}\right)\notag
\end{align}
for all $(q,x_1,\dotsc,x_m,y_2,y_4,\dotsc,y_{2\left\lfloor \tfrac{m}{2} \right\rfloor},y_m) \in \mathbb{R}^n \times \mathbb{R}^m \times \mathbb{R}^{\left\lceil \tfrac{m}{2} \right\rceil}$ and with $w_d \in \mathbb{R}[q]_{2d}$, $\sigma_0, \sigma_{x_k} \in \Sigma[q,x_1,\dotsc,x_m,\allowbreak y_2,y_4,\dotsc,y_m]_{2d}$ for $k=1,\dotsc,m-1$, $\sigma_k \in \Sigma[q,x_1,\dotsc,x_m,\allowbreak y_2,y_4,\dotsc,y_m]_{2d-2}$ for $k = 1, \dotsc, n$, $\tau_{\Re,k} \in \mathbb{R}[q,x_1,\dotsc,x_m,\allowbreak y_2,y_4,\dotsc,y_m]_{2d-k}$ for $k=1,\dotsc,m$ and $\tau_{\Im,k}$ for $k=\left\lfloor \tfrac{m}{2} \right\rfloor,\dotsc,m$.

\begin{rem}
As in Remark \ref{strength}, SDP \eqref{taud} is a strengthening of LP \eqref{tau}, meaning $\rho_d \leq \rho$. Also as in Remark \ref{qm}, the quadratic module corresponding to $\mathcal{Z}_o$ is archimedean, i.e. $\lim_{d\to\infty}\rho_d  = \rho$.
\end{rem}

We conclude the section with the following result:
\begin{thm} \label{converged}
	Let $w_d \in \mathbb{R}[q]_{2d}$ be a near optimal solution for SDP \eqref{taud}, i.e. $\int_\mathcal{Q} w_d(q)\, dq \geq \rho_d - \tfrac{1}{d}$ and consider the associated sequence $(w_d)_{d\geq d_0} \subset L^1(\mathcal{Q})$. Then $w_d$ converges to $a$ in $L^1$ norm in $\mathcal{Q}$.
\end{thm}
Unsurprisingly, one can prove this result in exactly the same way as Theorem \ref{converge}, so we do not detail the proof here. Remark that the first part of the proof can be shortened, since $\int_{\mathcal{Z}_o} x_m\, d\mu(q,x_1,\dotsc,x_m,\allowbreak y_2,y_4,\dotsc,y_m) = \int_{\mathcal{Z}_o} a(q)\, d\mu(q,x_1,\dotsc,x_m,\allowbreak y_2,y_4,\dotsc,y_m)$.

\subsubsection{Examples}

Just as the upper abscissa approximation automatically approximates the stability region from inside, the lower approximation gives, as a side effect, an outer approximation. In this section we will examine similar examples as for the upper approximation.

\begin{figure}[h]
\centerline{\includegraphics[width=0.5\textwidth]{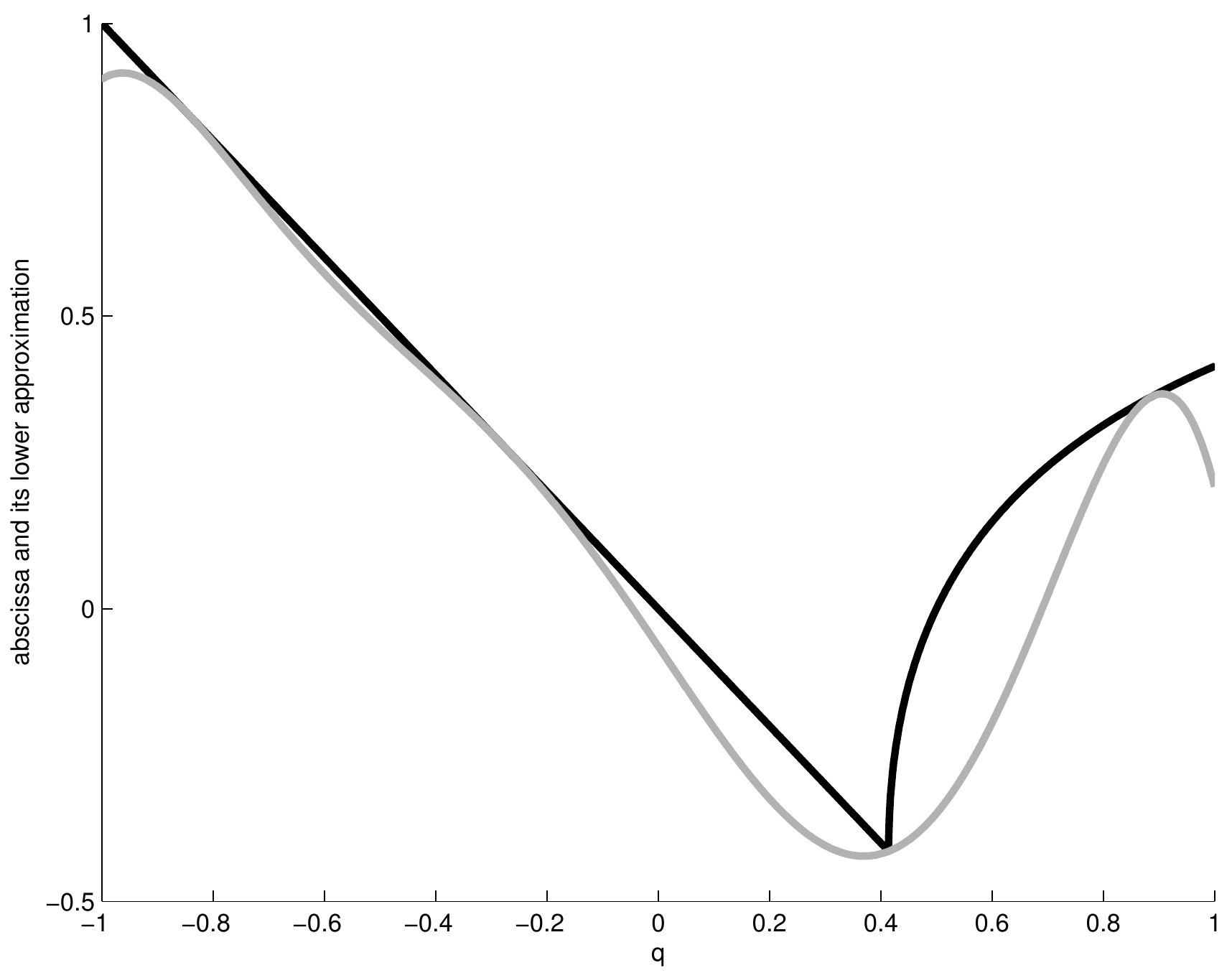}\includegraphics[width=0.5\textwidth]{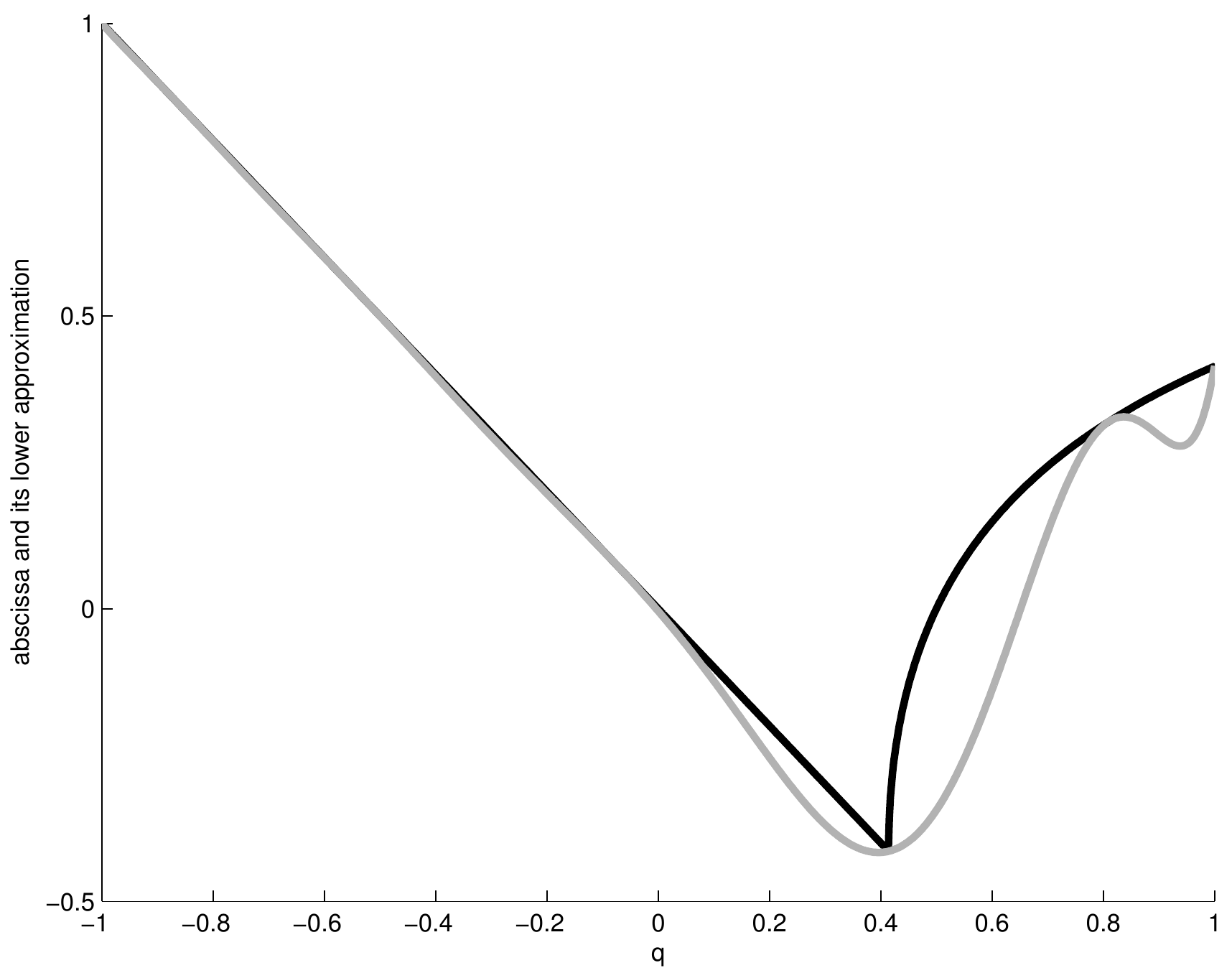}}
\caption{Abscissa (black) and its polynomial lower approximations of degree 6 (gray, left) and 10 (gray, right)
for Example \ref{expl1d}. The quality of the approximation deteriorates near the minimum, where the
abscissa is not Lipschitz, compare with Figure \ref{fig-expl1}.}\label{fig-expl1d}
\end{figure}

\begin{expl}\label{expl1d}
As in Example \ref{expl1} consider the polynomial
\begin{equation*}
	p:\ s \mapsto p(q,s) = s^2 + 2qs + 1-2q.
\end{equation*}
We have $y_1=-y_2$, so $\mathcal{Z}_o := \{(q,x_1,x_2,y_2) \in \mathcal{Q} \times \mathbb{R}^3 : x_1 \leq x_2, \ -2q = x_1 + x_2,\ 1-2q = x_1x_2 + y_2^2,\ 0 = (x_1-x_2)y_2\}$. 
In Figure \ref{fig-expl1d} we see the graphs of the degree 6 and 10 polynomial lower approximations obtained by
solving SDP  (\ref{taud}). As in Example \ref{expl-lower}, we observe that the SDP solver Mosek does not return
a correct degree 10 polynomial, and we had to use the SDP solver SeDuMi instead in this case.
Due to the rather big amount of variables and constraints, computing the degree 10 solution is already relatively expensive, with a few seconds of CPU time.
\end{expl}

\begin{figure}[h]
\centerline{\includegraphics[width=0.5\textwidth]{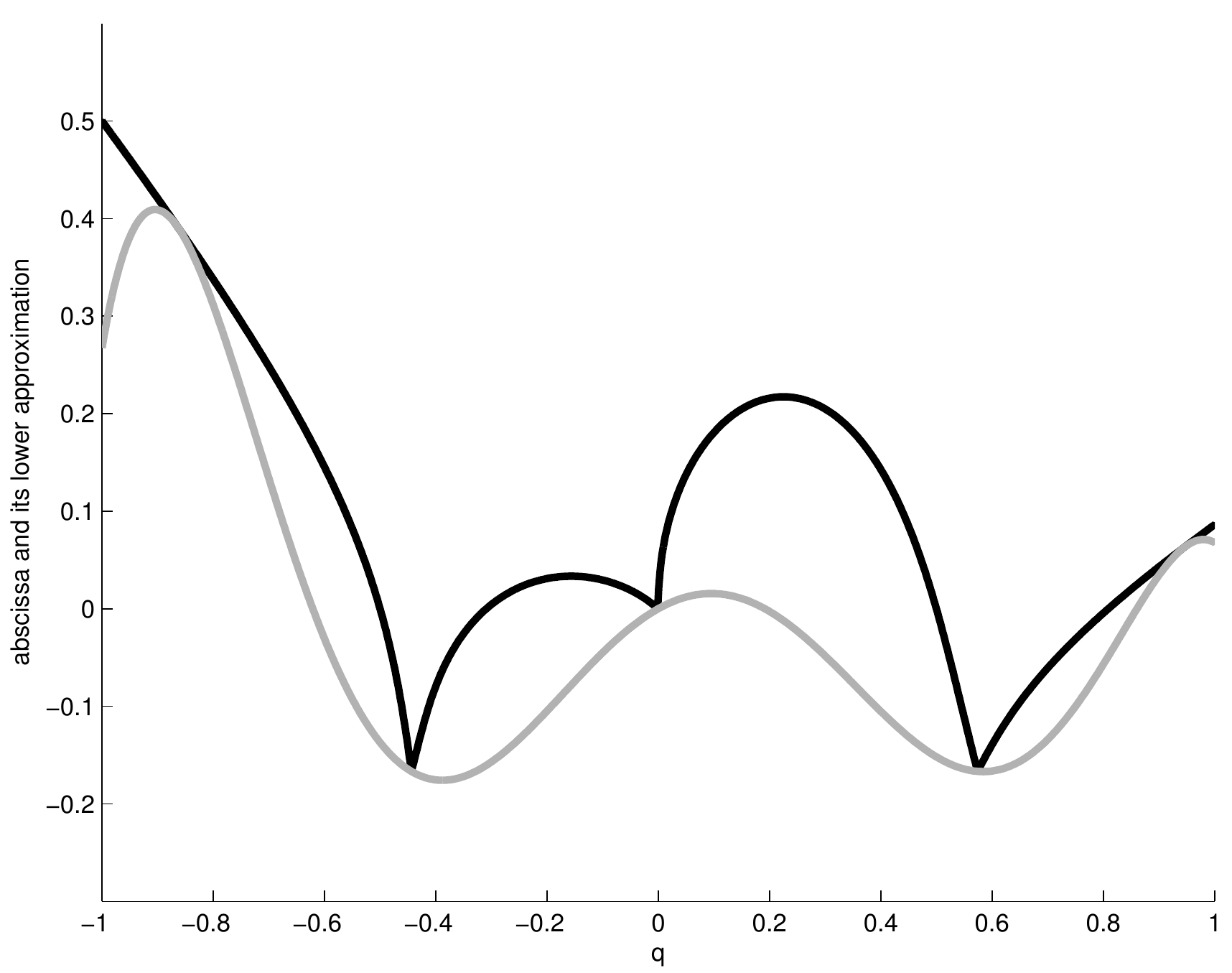}\includegraphics[width=0.5\textwidth]{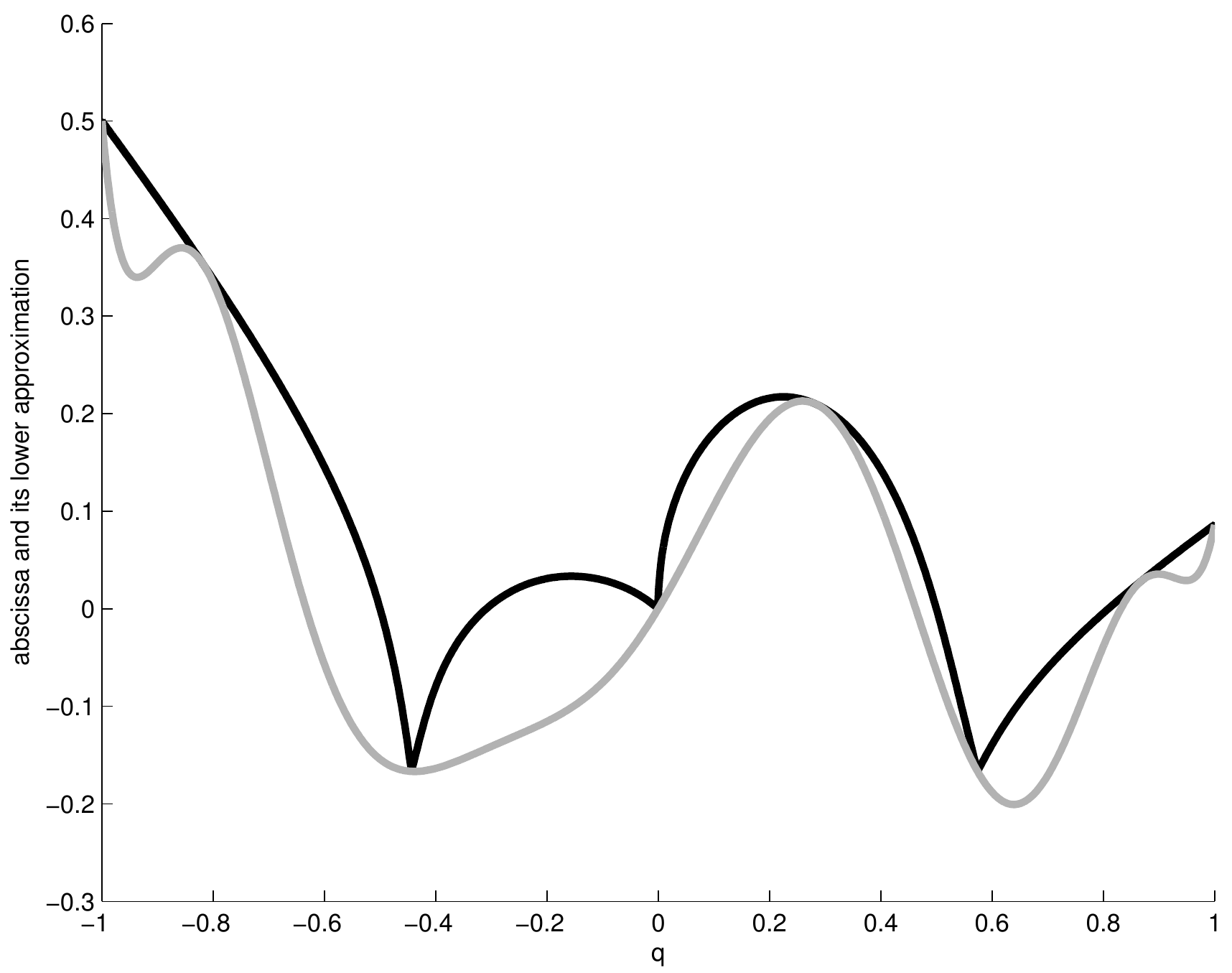}}
\caption{Abscissa (black) and its polynomial lower approximations of degree 6 (gray, left) and 10 (gray, right)
for Example \ref{expl2d}. The quality of the approximation deteriorates near the minimum, where the
abscissa is not differentiable, compare with Figure \ref{fig-expl2}.}\label{fig-expl2d}
\end{figure}

\begin{expl}\label{expl2d}
As in Example \ref{expl2} consider the polynomial
	\begin{equation*}
	p:\ s \mapsto p(q,s) = s^3 + \tfrac{1}{2}s^2 + q^2s + (q-\tfrac{1}{2})q(q+\tfrac{1}{2}).
\end{equation*}
With $y_1=-y_2-y_3$ we calculate $\mathcal{Z}_o$ as in Example \ref{explm3}. In Figure \ref{fig-expl2d} we see the graphs of the degree 6 and 10 polynomial lower approximations obtained by
solving SDP (\ref{taud}). The computation time to get the degree 10 solution is around 15 minutes, which
is arguably not a good compromise given the quality of the approximation.
\end{expl}

\begin{rem}
As for the upper abscissa approximation, we observe practically that  the implementation for the lower approximation is rather sensitive to polynomials with large coefficients.
\end{rem}
%
%
%
%

\begin{figure}[h]
\centerline{\includegraphics[width=0.5\textwidth,height=0.3\textheight]{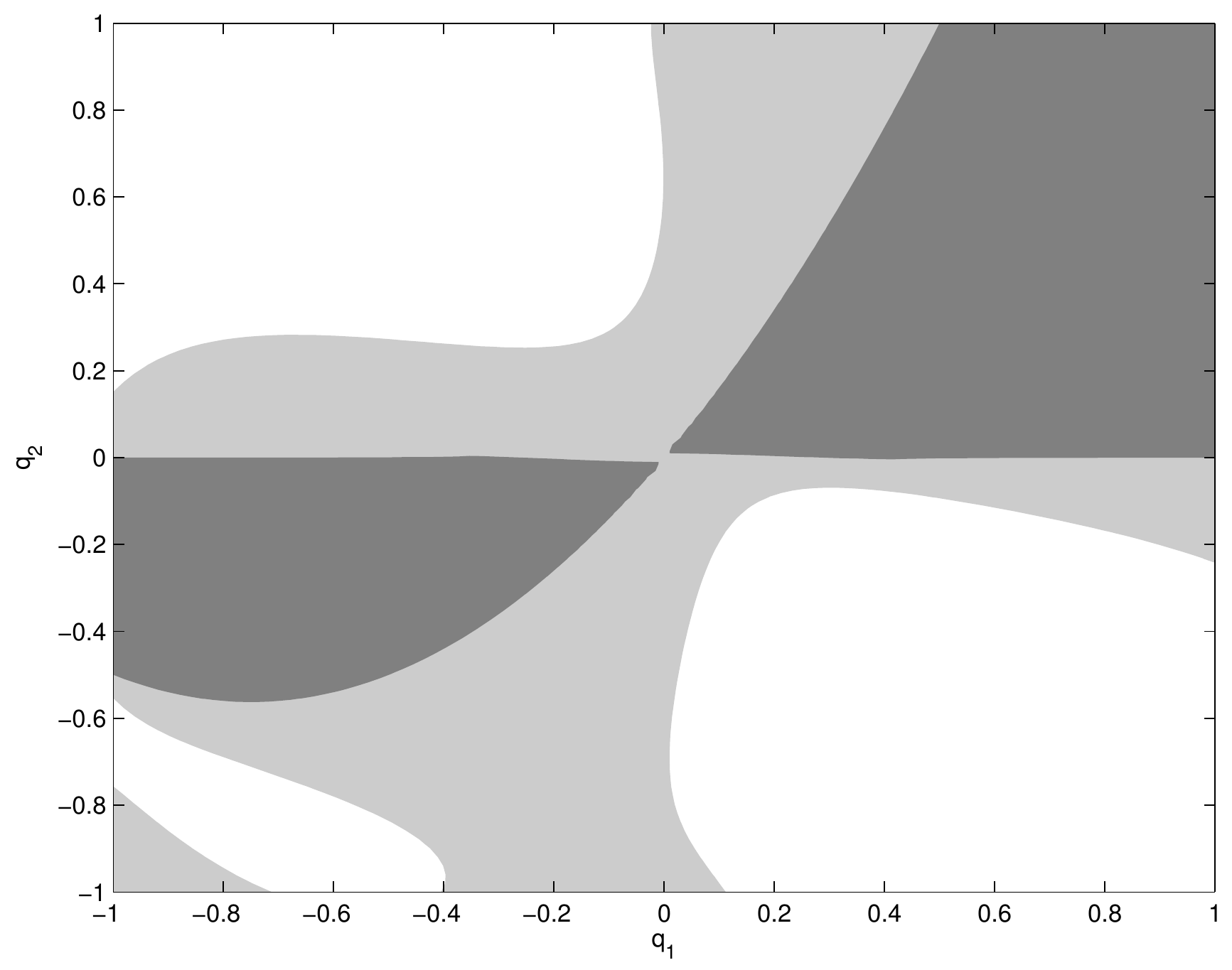}
\includegraphics[width=0.5\textwidth,height=0.3\textheight]{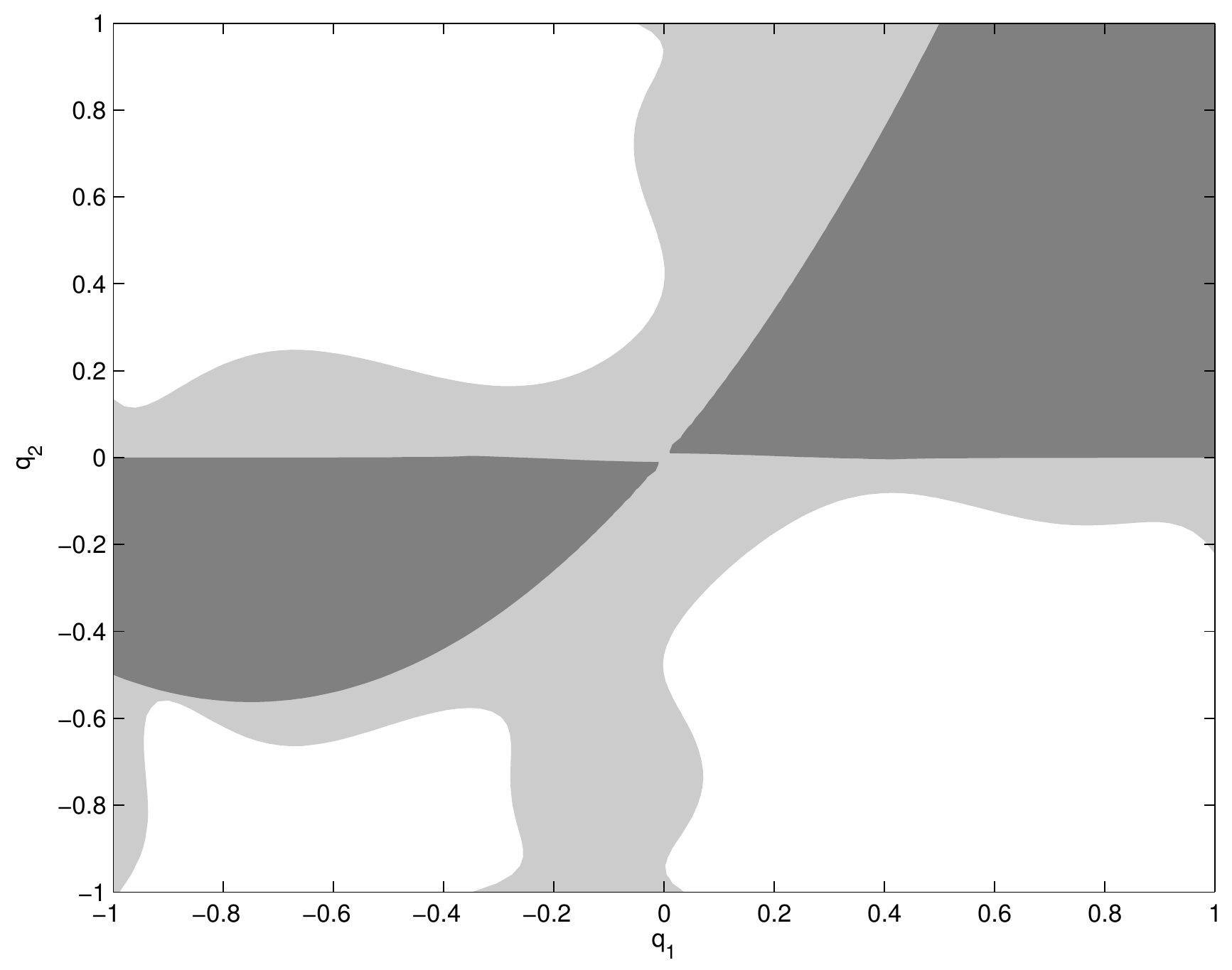}}
\caption{Stabilizability region (dark gray region) and its degree 6 outer approximation (light gray region, left) and degree 10 outer approximation (light gray region, right). Compare with Figure \ref{fig-expl3bis}.}\label{fig-expl3d}
\end{figure}

\begin{expl}\label{expl3d}
	As in Example \ref{expl3}, consider the polynomial 
	\begin{equation*}
		p:\ s \mapsto p(q,s) = s^3 + (q_1+\tfrac{3}{2}) s^2 + q_1^2s + q_1q_2.
	\end{equation*}	
	Since we have $m=3$, the set $\mathcal{Z}_o $ is again given in Example \ref{explm3}. 
In Figure \ref{fig-expl3d} we see the outer approximation of degrees $6$ and $10$ obtained by solving SDP (\ref{taud}).
We notice an opening in the approximation of the stability region in the lower half of the picture. This is due to $a$ being zero and non-smooth for $q_1 = 0$, meaning $a(0,q_2)=0$. This phenomenon also incapacitates $w_8$ to get tighter to $a$ for $q_2 > 0$ than we observe in the upper half of the picture. 
\end{expl}

\subsection{Lower approximation via Gau{\ss}-Lucas}\label{lowergl}

\subsubsection{Problem formulation}

As indicated above, we want to find a semi-algebraic subset of $\mathcal{Z}$ which contains only those roots of $p$ whose real part is maximal. This means that, in contrast to the approach of Section \ref{esf},
we will not rephrase $\mathcal{Z}$, but formulate further constraints.

In order to do this we must distinguish between the roots of $p(q,\cdot)$ according to the size of their real parts. For this purpose we use the following result:

\begin{thm}[Gau{\ss}-Lucas]\label{gl}
The critical points of a non-constant polynomial lie in the convex hull of its roots.
\end{thm}

We refer to \cite{gausslucas_pisa} for further information and a proof. Let us denote the derivative of $p(q,s)$ with respect to $s$ by $p'(q,s)$. By Theorem \ref{gl}, the roots of $p'(q,\cdot)$ are contained in the convex hull of the roots of $p(q,\cdot)$. It follows readily that the abscissa $a_{p'}$ of $p'$ lies below the abscissa $a_p$ of $p$:
\begin{equation*}
	a_{p'}(q) \leq a_p(q)\ \text{for all } q \in \mathcal{Q}.
\end{equation*}

However, $p$ may have some roots with real part strictly smaller than $a_p$ and strictly bigger than $a_{p'}$, meaning that the root whose real part is the abscissa is not the only one whose real part lies above $a_{p'}$. Of course, this cannot happen for real polynomials $\mathbb{R} \to \mathbb{R}$ because of monotonicity, and neither for complex polynomials of degree 2. But, for example, for $n=1$ the polynomial $p(q,s) = s^4 + (q^2+1)s + q$ has two roots with different real parts greater than $a_{p'}$ for $q \in [-1,-0.4]$.

To prevent the lower abscissa approximation from converging to the real part of a root smaller than the abscissa, we make the following assumption:

\begin{ass}\label{ass1}
	None of the real parts of any root of $p$ lies strictly between $a_p$ and $a_{p'}$, i.e. $x \notin {]a_{p'}(q), a_p(q)[}$ for all 
$(q,x,y) \in \mathcal{Z}$.
\end{ass}

\begin{rem}
Unfortunately, we do not know how restrictive this assumption is. For $n=1$ it was rather difficult to find examples that violate it.
\end{rem}

Now let $\hat{v} \in \mathscr{C}(\mathcal{Q})$ be a near optimal solution to  LP \eqref{v} for the polynomial $p'$, meaning $\int_\mathcal{Q} \hat{v}(q)\, dq \leq \rho + \varepsilon$ for an $\varepsilon > 0$. Then, $\hat{v}$ is an upper approximation of the abscissa $a_{p'}$ of $p'$. We define the following subset:
\begin{equation*}
	\mathcal{Z}_{r} := \{(q,x,y) \in \mathcal{Z} \mid x - \hat{v}(q) \geq 0\}.
\end{equation*}
In order to see where we are going, let us pretend for a moment that we have an optimal solution. Then, under Assumption \ref{ass1}, the set $\mathcal{Z}_r$ would contain exactly the points $(q,a(q),y_q)$ with $y_q$ denoting the imaginary part of the root of $p(q,\cdot)$ with maximal real part. Hence, the solution to the following LP would give a lower approximation of the abscissa function $a_p$ of $p$:
\begin{align}
	& \sup_{w \in \mathscr{C}[q]}  \int_\mathcal{Q} w(q)\, dq \label{wv}\\
	& \text{s.t.}\quad x - w(q) \geq 0 \text{ for all } (q,x,y) \in \mathcal{Z}_r. \notag
\end{align}
Since $\hat{v}$ might not be optimal, the projection of $\mathcal{Z}_r$ onto $\mathcal{Q}$ can have holes of volume $\varepsilon$. As a consequence, $w$ might not be a valid lower bound of the abscissa on these holes.\par\smallskip

Taking this into account, we build an SDP hierarchy for LP \eqref{wv} in the next section. The issue is that we have to consider the hierarchy for the upper approximation of $a_{p'}$ first and the solution to it might interfere with $a_p$.

\subsubsection{SDP hierarchy}

For $d'_0 \in \mathbb{N}$ sufficiently large we denote by $\hat{v}_{d'}$, $d' \geq d'_0$, the solutions to SDP \eqref{vd} for the polynomial $p'$. Thus, the $\hat{v}_{d'}$ are polynomials in $\mathbb{R}[q]_{2d'}$ and by Theorem \ref{converge} the sequence $(\hat{v}_{d'})_{d' \in \mathbb N}$ converges to $a_{p'}$ from above in $L^1$ norm.

Next, we want to describe the set $\mathcal{Z}_r$ via the polynomials $\hat{v}_{d'}$ in order to have an implementable problem, i.e. we define
\begin{equation*}
	\mathcal{Z}_{r,d'} := \{(q,x,y) \in \mathcal{Z} : x - \hat{v}_{d'}(q) \geq 0\}.
\end{equation*}
Of course, the set $\mathcal{Z}_{r,d'}$ is highly dependent on the quality of $\hat{v}_{d'}$ and hence on the choice of $d'$. Evidently, $\mathcal{Z}_{r,d'}$ is a subset of $\mathcal{Z}_r$, possibly strictly. To ensure that $\mathcal{Z}_{r,d'}$ contains all roots of $p$ with the abscissa as their real parts we need $\hat{v}_{d'} \leq a_p$. However, in practice this is impossible in some cases:

\begin{expl}
The abscissa $a_p$ of $p(q,s) = (s^3+q)^2$ and the abscissa $a_{p'}$ of $p'$ coincide and have a point of non-differentiability at $q=0$. As another example consider the polynomial $p(q,s) = s^4 + qs$ for which both $a_p$ and $a_{p'}$ are not differentiable at $q = 0$ and $a_p(0) = a_{p'}(0) = 0$.
\end{expl}

For these examples we cannot achieve $\hat{v}_{d'} \leq a_p$ with $d'$ finite, since $\hat{v}_{d'}$ is a polynomial and therefore differentiable everywhere.

As a consequence, we  formulate another assumption. In general, the points that may cause problems are the ones where $a_p$ and $a_{p'}$ coincide, i.e. the points of the set
\begin{equation*}
	\mathcal{D} := \{q \in \mathcal{Q} \mid a_p(q) = a_{p'}(q)\}.
\end{equation*}
On this set the polynomial $\hat{v}_{d'}$ should approximate $a_{p'}$ perfectly for a finite $d'$. Calling a solution $\hat{v}_{d'}$ near optimal if it satisfies $\int_\mathcal{Q} \hat{v}_{d'}(q)\, dq \leq \rho_{d'} + \tfrac{1}{d'}$, we assume:

\begin{ass}\label{ass2}
There is a near optimal solution $\hat{v}_{d'}$ to SDP \eqref{vd} for the polynomial $p'$ with $d'$ finite such that $\hat{v}_{d'}$ and $a_{p'}$ coincide on $\mathcal{D}$.
\end{ass}

\begin{rem}
	A sufficient condition for a violation of Assumption \ref{ass2} is the existence of a value of $q$ for which $a_{p'}$ is not differentiable and $a_p(q) = a_{p'}(q)$.
	This is the case for the examples given above. Note also that they are of degenerate nature.
\end{rem}

	To face another issue, we denote the projection of $\mathcal{Z}_{r,d'}$ onto the set $\mathcal{Q}$ by $\pi_\mathcal{Q}(\mathcal{Z}_{r,d'})$, i.e.
\begin{equation*}
	\pi_\mathcal{Q}(\mathcal{Z}_{r,d'}) = \{q \in \mathcal{Q} : \exists x,y \in \mathbb{R}: (q,x,y) \in \mathcal{Z}_{r,d'}\}.
\end{equation*}
Since $\hat{v}_{d'}$ converges to $a$ in $L^1$, but not necessarily uniformly, it might have spikes or similar irregularities, meaning that the set $\mathcal{Q} \setminus \pi_\mathcal{Q}(\mathcal{Z}_{r,d'})$ is not empty.
However, the $L^1$ convergence of $\hat{v}_{d'}$, or more precisely the convergence in measure, implies that there is a subsequence $(\hat{v}_{d'_l})_{l \in \mathbb{N}}$ which converges to $a_{p'}$ almost uniformly (see e.g. \cite[Theorem 2.5.3]{ash}). In other words, for all $\delta >0$, there exists a set $\mathcal{A}_\delta$ in the Borel sigma algebra of $\mathcal{Q}$ such that $\int_{\mathcal{A}_\delta} dq < \delta$ and $\hat{v}_{d'_l}$ converges  uniformly on $\mathcal{A}_\delta^C$ to $a_{p'}$ when $l\to\infty$, where $\mathcal{A}_\delta^C$ is the set-theoretic complement of $\mathcal{A}_\delta$ in $\mathcal{Q}$. With this notation we have
\begin{equation*}
	\pi_\mathcal{Q}(\mathcal{Z}_{r,d'}) \subseteq \mathcal{A}_\delta^C \subseteq \mathcal{Q}.
\end{equation*}
\begin{lem}\label{v_p'}
Let Assumption \ref{ass2} hold. Then, for every $\delta > 0$ there is a finite $d' \in \mathbb N$
such that $\hat{v}_{d'} \leq a_p$ on $\mathcal{A}_\delta^C$.
\end{lem}
\begin{proof} Fix $\delta > 0$. Obviously we want
	\begin{equation}\label{a}
		0 \leq a_p(q) - \hat{v}_{d'}(q) = a_p(q) - a_{p'}(q) + a_{p'}(q) - \hat{v}_{d'}(q)
	\end{equation}	
	for every $q \in \mathcal{A}_\delta^C \subseteq \mathcal{Q}$. By Theorem \ref{gl}, we have $a_p(q) - a_{p'}(q) \geq 0$ for all $q \in \mathcal{Q}$. Otherwise, the difference $a_{p'}(q) - \hat{v}_{d'}(q)$ is negative by construction, but due to Theorem \ref{converge} we find a subsequence $\hat{v}_{d'_l}$ converging uniformly to $a_{p'}$ on $\mathcal{A}_\delta^C$.
	Hence, there is a finite $d'_{l^*}$ such that \eqref{a} is fulfilled for all $q \in \{q \in \mathcal{A}_\delta^C : a_p(q) > a_{p'}(q)\}$. Because of Assumption \ref{ass2} there is also a finite $d'$ such that $a_{p'}(q) - \hat{v}_{d'}(q)$ vanishes on $\{q \in \mathcal{A}_\delta^C : a_p(q) = a_{p'}(q)\}$. Taking $d'_{l'} \geq d'$ with $l' \geq l^*$ completes the proof. 
\end{proof}
\begin{rem}
Choosing $d'$ according to Lemma \ref{v_p'} implies $\pi_\mathcal{Q}(\mathcal{Z}_{r,d'}) = \mathcal{A}_\delta^C$.
\end{rem}

Under Assumption \ref{ass1} and \ref{ass2} and with an appropriate choice of $d'$ (depending on $\delta$) the solution to the following LP gives a lower approximation of the abscissa function $a_p$ of $p$ on the set $\mathcal{A}_\delta^C \subseteq \mathcal{Q}$:
\begin{align}
	\rho_{d'} = & \sup_{w \in \mathscr{C}[q]}  \int_{\pi_\mathcal{Q}(\mathcal{Z}_{r,d'})} w(q)\, dq \label{wvd}\\
	& \text{s.t.}\quad x - w(q) \geq 0 \text{ for all } (q,x,y) \in \mathcal{Z}_{r,d'}. \notag
\end{align}
\begin{rem}
	Note that under Assumption \ref{ass1}, LP \eqref{wvd} always provides a proper approximation for the abscissa $a_p$ from below on $\pi_\mathcal{Q}(\mathcal{Z}_{r,d'})$, but this might not be very useful, since for bad $\hat{v}_{d'}$ this set may have big holes or even be empty. To achieve suitable results on $\mathcal{A}_\delta^C$ we need Assumption \ref{ass2} and an appropriate $d'$, meaning a sufficiently good $\hat{v}_{d'}$ ensuring $\pi_\mathcal{Q}(\mathcal{Z}_{r,d'}) = \mathcal{A}_\delta^C$. 
\end{rem}	
In analogy with \eqref{mu}, the dual LP reads
\begin{align}
	\rho_{d'}^* = & \inf_{\mu \in \mathscr{M}^+(\mathcal{Z}_{r,d'})}  \int_{\mathcal{Z}_{r,d'}} x\ d\mu(q,x,y) \label{mugl}\\
	& \text{s.t.}  \int_{\mathcal{Z}_{r,d'}} q^\alpha\, d\mu = \int_{\pi_\mathcal{Q}(\mathcal{Z}_{r,d'})} q^\alpha\, dq, \text{ for all } \alpha \in \mathbb{N}^n \notag
\end{align}
with the notation of Section \ref{sec1}.
\begin{lem}\label{zerogl}
The infimum in LP \eqref{mugl} is attained, and there is no duality gap
between LP \eqref{wvd} and LP \eqref{mugl}, i.e. $\rho_{d'} = \rho_{d'}^*$.
\end{lem}
Since $\mathcal{Z}_{r,d'}$ is a compact subset of $\mathcal{Z}$, we can mimic the proof of Lemma \ref{zero} in order to get a proof of Lemma \ref{zerogl}.
\begin{rem}
As in Remark \ref{attained}, the supremum in LP \eqref{wvd} is not attained for $\mathscr{C}(\mathcal{Q})$ or $\mathbb{R}[q]$, but it is attained for $\mathbb{R}[q]_d$ with $d$ finite. To adjust the proof of Remark \ref{attained}, set $M := \min_{q \in \mathcal{Q}} a(q) - N$ for an $N \in \mathbb{N}$ sufficiently large, and $R := \int_{\mathcal{Q}} \left(a(q)-M\right)\, dq$ as in Remark \ref{attainedd}.
\end{rem}
Finally, for $d'$ as in Lemma \ref{v_p'} and $d_0 \geq d'$ sufficiently large we can write an SDP hierarchy indexed by $d \in \mathbb{N},\ d \geq d_0$:
\begin{align}
		\rho_{d',d} = & \sup_{w_d,\sigma_0,\sigma_{j},\sigma_{\hat{v}},\tau_\Re,\tau_\Im} \int_{\pi_\mathcal{Q}(\mathcal{Z}_{r,d'})} w_d(q)\, dq\label{wdvd}\\
		& \text{s.t.}\ x - w_d(q) = \sigma_0(q,x,y) + \sum_{j=1}^n \sigma_{j}(q,x,y)(1-q_j^2) + \sigma_{\hat{v}}(q,x,y)(x-\hat{v}_{d'}(q))\notag\\
		&\hspace{2,5cm} + \tau_\Re(q,x,y)p_\Re(q,x,y) + \tau_\Im(q,x,y)p_\Im(q,x,y)\notag
	\end{align}
for all $(q,x,y) \in \mathbb{R}^n \times \mathbb{R}^2$ and with $w_d \in \mathbb{R}[q]_{2d}$, $\sigma_0 \in \Sigma[q,x,y]_{2d},\ \sigma_{j} \in \Sigma[q,x,y]_{2d-2}$ for $j = 1, \dotsc, n$, $\sigma_{\hat{v}} \in \Sigma[q,x,y]_{2d-d'}$ and $\tau_\Re, \tau_\Im \in \mathbb{R}[q,x,y]_{2d-m}$.
\begin{rem}
As in section \ref{las}, SDP \eqref{wdvd} is a strengthening of LP \eqref{wvd}, meaning $\rho_{d',d} \leq \rho_{d'}$. Besides, the archimedean quadratic module corresponding to the set $\mathcal{Z}$ is contained in the quadratic module corresponding to $\mathcal{Z}_{r,d'}$. Hence, the latter is also archimedean, i.e. $\lim_{d\to\infty} \rho_{d',d} = \rho_{d'} = \rho_{d'}^*$.
\end{rem}
\begin{rem}
	For numerical applications one can assume that $\mathcal{A}_\delta$ is empty and substitute $\pi_\mathcal{Q}(\mathcal{Z}_{r,d'})$ by $\mathcal{Q}$.
\end{rem}
The associated sequence converges:
\begin{thm} \label{convergegl}
	Let Assumptions \ref{ass1} and \ref{ass2} hold and let $\mathcal{A}_\delta^C$ and $d'$ be as in Lemma \ref{v_p'}. Let $w_d \in \mathbb{R}[q]_{2d}$ be a near optimal solution for SDP \eqref{wdvd}, i.e. $\int_\mathcal{Q} w_d(q)\, dq \geq \rho_{d,d'} - \tfrac{1}{d}$. Consider the associated sequence $(w_d)_{d\geq d_0} \subset L^1(\mathcal{Q})$. Then $w_d$ is a valid lower bound of $a_p$ on $\mathcal{A}_\delta^C$ and it converges to $a_p$ in $L^1$ norm on $\mathcal{A}_\delta^C$.
\end{thm}
The proof of this result is very similar to the proof of Theorem \ref{converge}, so we omit it. Note that by Lemma \ref{v_p'} every feasible solution to SDP \eqref{wdvd} is a valid lower bound of $a_p$ on $\mathcal{A}_\delta^C$ and that we have $\pi_\mathcal{Q}(\mathcal{Z}_{r,d'}) = \mathcal{A}_\delta^C$ due to our choice of $d'$. As for the proof of Theorem \ref{converged}, the first part can be shortened, since for every $(q,x,y) \in \mathcal{Z}_{r,d'}$ it holds that $x = a(q)$.

\subsubsection{Examples}

\begin{figure}[h]
\centerline{\includegraphics[width=0.5\textwidth]{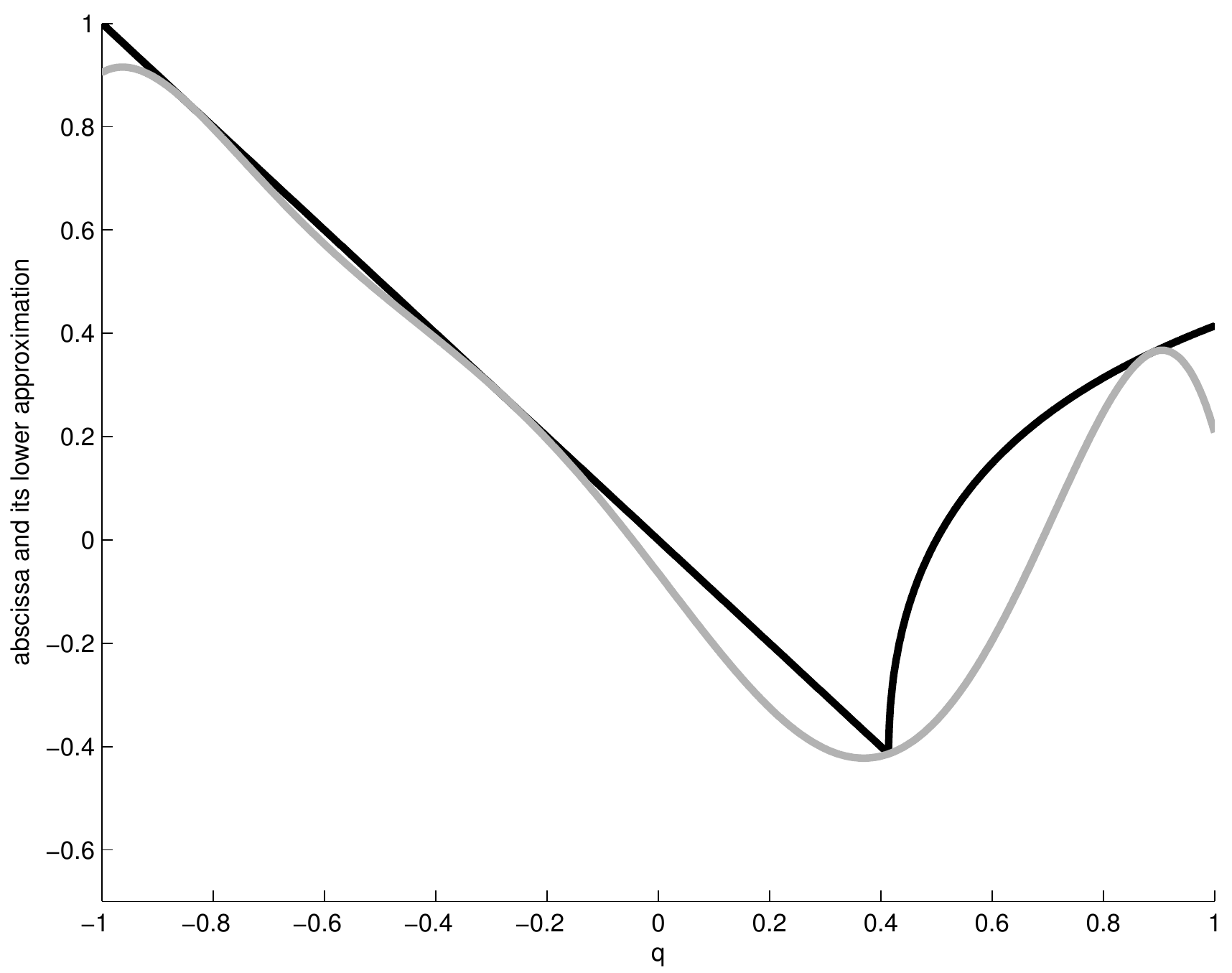}\includegraphics[width=0.5\textwidth]{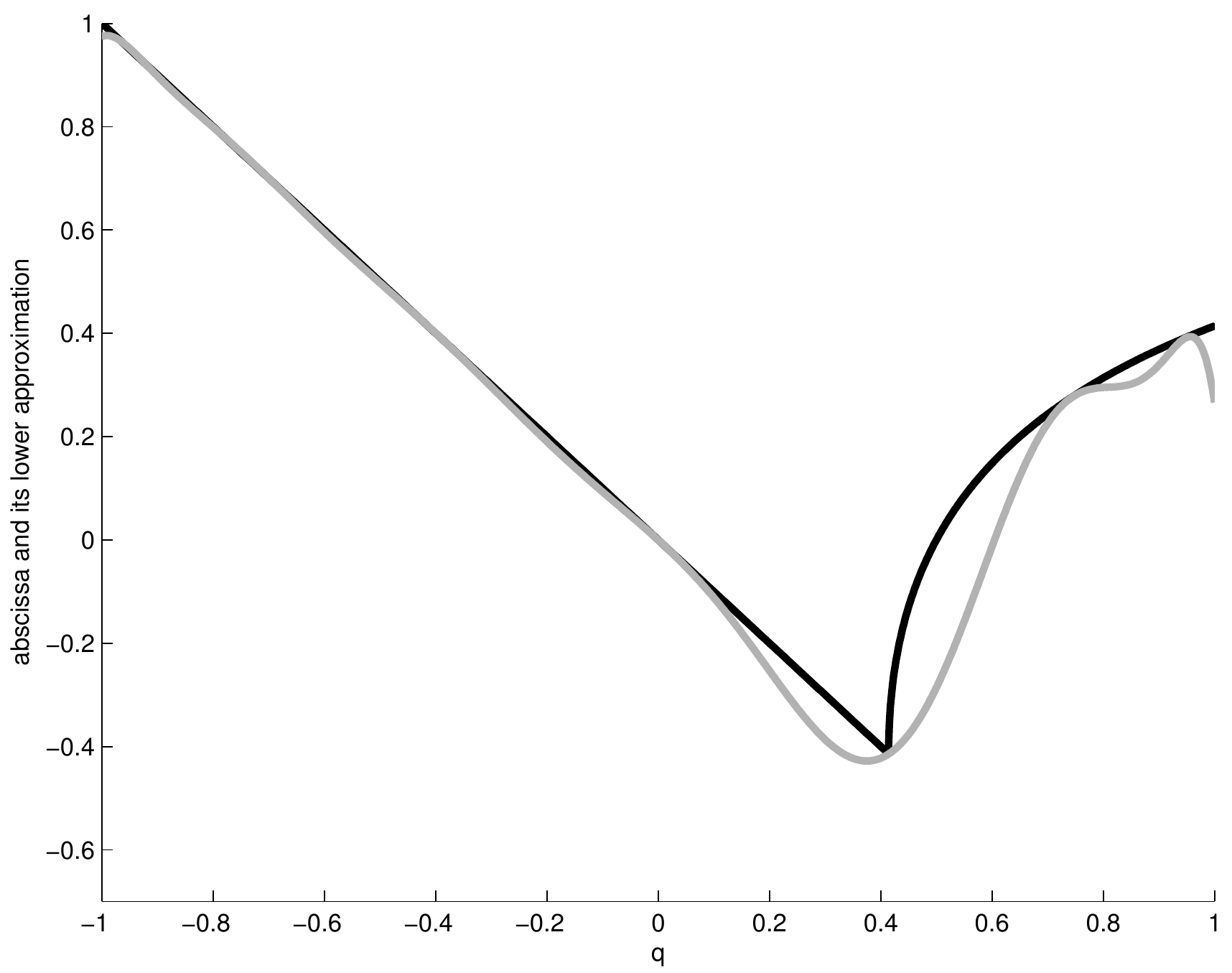}}
\caption{Abscissa (black) and its polynomial lower approximations of degree 6 (gray, left) and 12 (gray, right)
for Example \ref{expl1gl}. The quality of the approximation deteriorates near the minimum, where the
abscissa is not Lipschitz, compare with Figures \ref{fig-expl1} and \ref{fig-expl1d} .}\label{fig-expl1gl}
\end{figure}

\begin{expl}\label{expl1gl}
As in Examples \ref{expl1} and \ref{expl1d} consider
\begin{equation*}
	p:\ s \mapsto p(q,s) = s^2 + 2qs + 1-2q.
\end{equation*}
Assumption \ref{ass1} is naturally fulfilled, since $p$ is of degree 2. In the same way, Assumption \ref{ass2} is fulfilled, since $a_{p'}(q) = -q$ is polynomial. We have $\mathcal{Z}_{r,d'} = \{(q,x,y) \in [-1,1] \times \mathbb{R}^2 : x - \hat{v}_{d'}(q) \geq 0,\ x^2-y^2+2qx+1-2q = 2xy+2qy = 0\}$ and the corresponding SDP \eqref{vd} reads
	\begin{align*}
		\rho_{d',d} = & \sup_{w_d,\sigma_0,\sigma_1,\sigma_{\hat{v}},\tau_\Re,\tau_\Im} \int_{-1}^1 w_d(q)\, dq\\
		& \text{s.t.}\ x - w_d(q) = \sigma_0(q,x,y) + \sigma_1(q,x,y)(1-q^2) + \sigma_{\hat{v}}(q,x,y)(x-\hat{v}_{d'}(q))\\
		&\hspace{2,5cm} + \tau_\Re(q,x,y)(x^2-y^2+2qx+1-2q) + \tau_\Im(q,x,y)(2xy+2qy)
	\end{align*}
for all $(q,x,y) \in \mathbb{R}^3$ and with $w_d \in \mathbb{R}[q]_{2d}$, $\sigma_0 \in \Sigma[q,x,y]_{2d},\ \sigma_1 \in \Sigma[q,x,y]_{2d-2}$, $\sigma_{\hat{v}} \in \Sigma[q,x,y]_{2d-d'}$ and $\tau_\Re, \tau_\Im \in \mathbb{R}[q,x,y]_{2d-2}$.
	Due to the simplicity of $a_{p'}$ it suffices to choose $d'=2$. We see the degree 6 and 12 polynomial lower approximations in Figure \ref{fig-expl1gl}. They are both computed in less than 2 seconds.
\end{expl}

\begin{figure}[h]
\centerline{\includegraphics[width=0.5\textwidth]{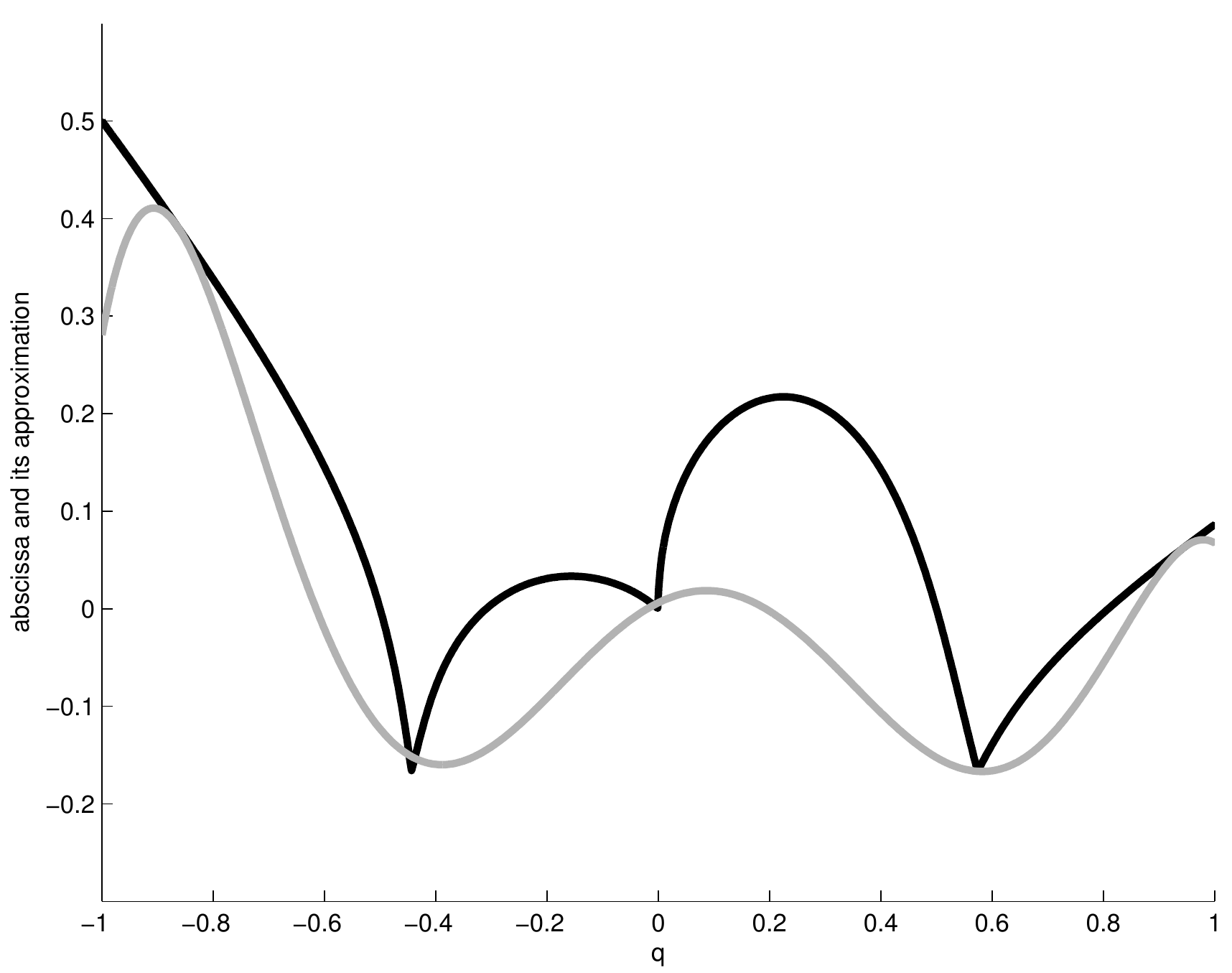}\includegraphics[width=0.5\textwidth]{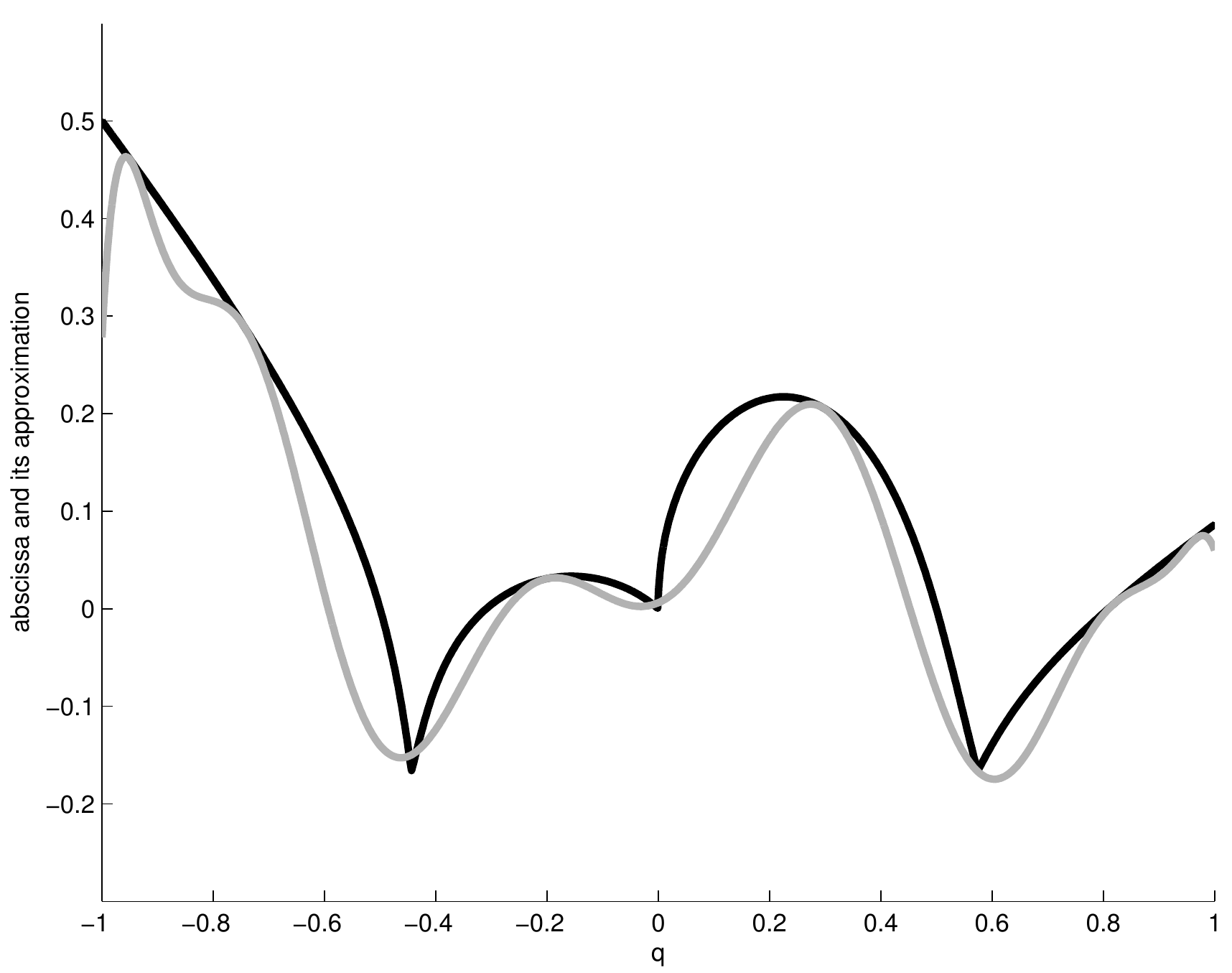}}
\caption{Abscissa (black) and its polynomial lower approximations of degree 6 (gray, left) and 12 (gray, right)
for Example \ref{expl2gl}. We observe that the approximations are not valid near $q=-0.5$ and $q=0$, as Assumption \ref{ass2} is violated.}\label{fig-expl2gl}
\end{figure}

\begin{expl}\label{expl2gl}
	As in Examples \ref{expl2} and \ref{expl2d} consider
	\begin{equation*}
	p:\ s \mapsto p(q,s) = s^3 + \tfrac{1}{2}s^2 + q^2s + (q-\tfrac{1}{2})q(q+\tfrac{1}{2}).
\end{equation*}
The abscissa $a_{p'}$ of $p'$ is not differentiable in two points, hence it is not a polynomial and it cannot be described perfectly by $\hat{v}_{d'}$ for finite $d'$. Let us choose $d'=8$ and $d=6$ resp. $d=12$. We observe in Figure \ref{fig-expl2gl} that $w_6$ resp. $w_{12}$ is not everywhere a valid lower bound. Indeed, the set $\mathcal{D} = \{q \in \mathcal{Q} \mid a_p(q) = a_{p'}(q)\}$ contains three points and for two of these (near $q=-0.5$ and $q=0$), the approximation $\hat{v}_8$ is not tight enough
to ensure $\pi_{\mathcal{Q}}(\mathcal{Z}_{r,8}) = \mathcal{Q}$. Consequently, Assumption \ref{ass2} is violated.
\end{expl}

\begin{figure}[h]
\centerline{\includegraphics[width=0.5\textwidth]{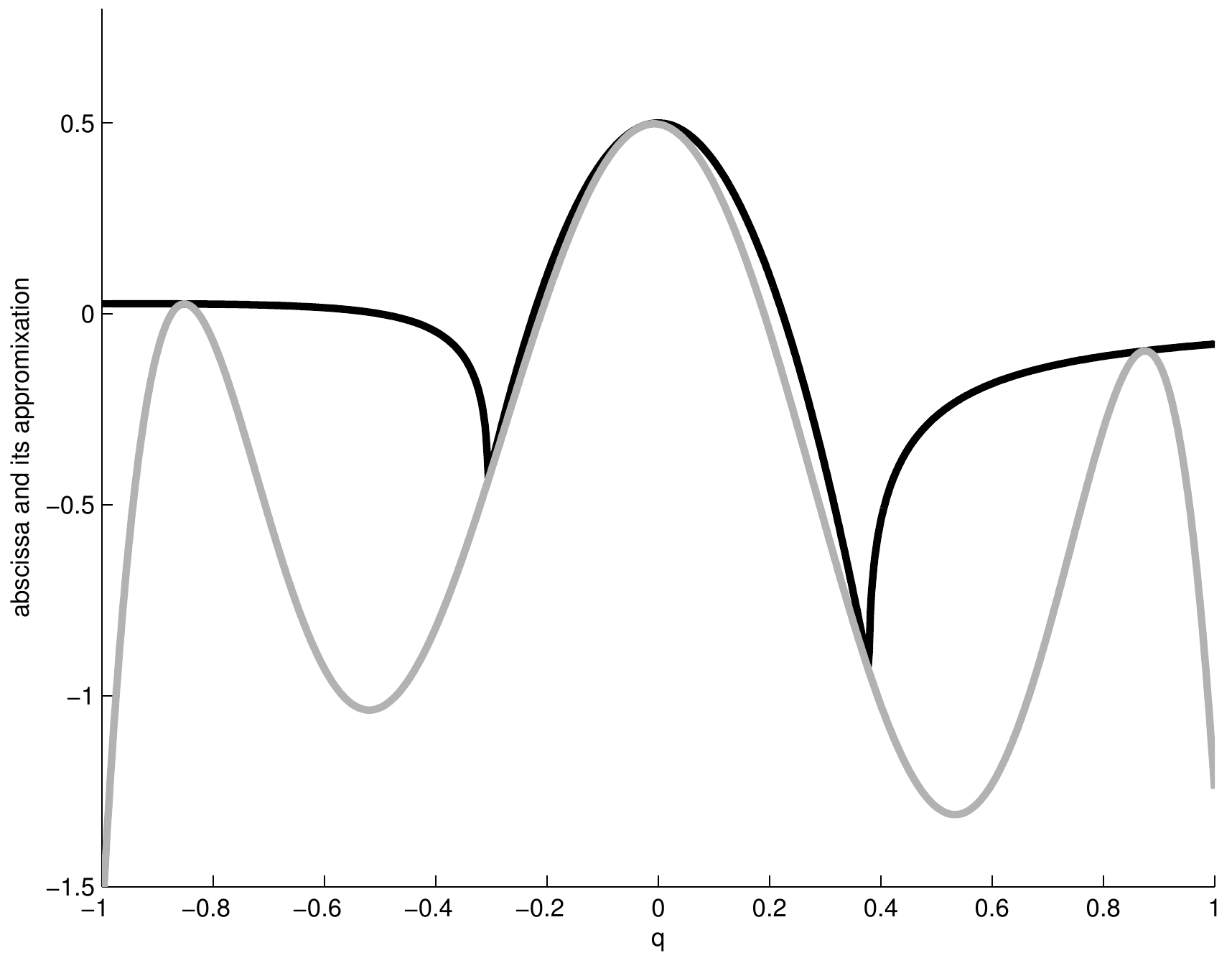}\includegraphics[width=0.5\textwidth]{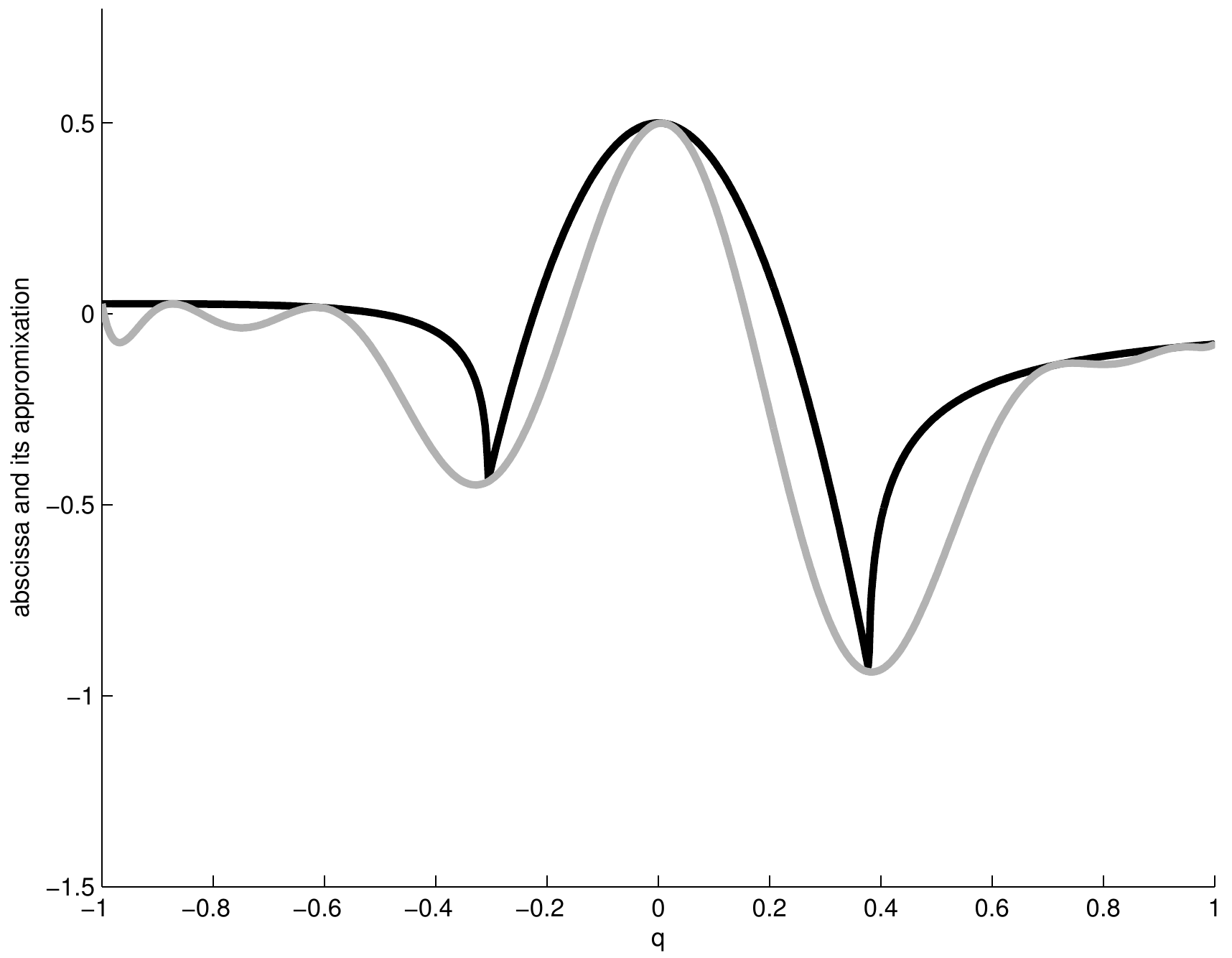}}
\caption{Abscissa (black) and its polynomial lower approximations of degree 6 (gray, left) and 12 (gray, right)
for Example \ref{explgl1}.}\label{fig-explgl1}
\end{figure}

\begin{expl}\label{explgl1}
	In order to discuss another example for which $\mathcal{D}$ is a non-empty interval, consider the polynomial
	\begin{equation*}
		p:\ s \mapsto p(q,s) = s^2 + (20q^2-1)s + q+\tfrac{1}{2}.
	\end{equation*}
	Here $a_{p'}(q) = -10q^2 + \tfrac{1}{2}$ is a quadratic polynomial. Thus, Assumption \ref{ass2} is fulfilled, in particular $\hat{v}_2 = a_{p'}$, and the lower approximations are valid, see Figure \ref{fig-explgl1}.
%
\end{expl}
%
%
%
%
%

\begin{figure}[h]
\centerline{\includegraphics[width=0.5\textwidth,height=0.3\textheight]{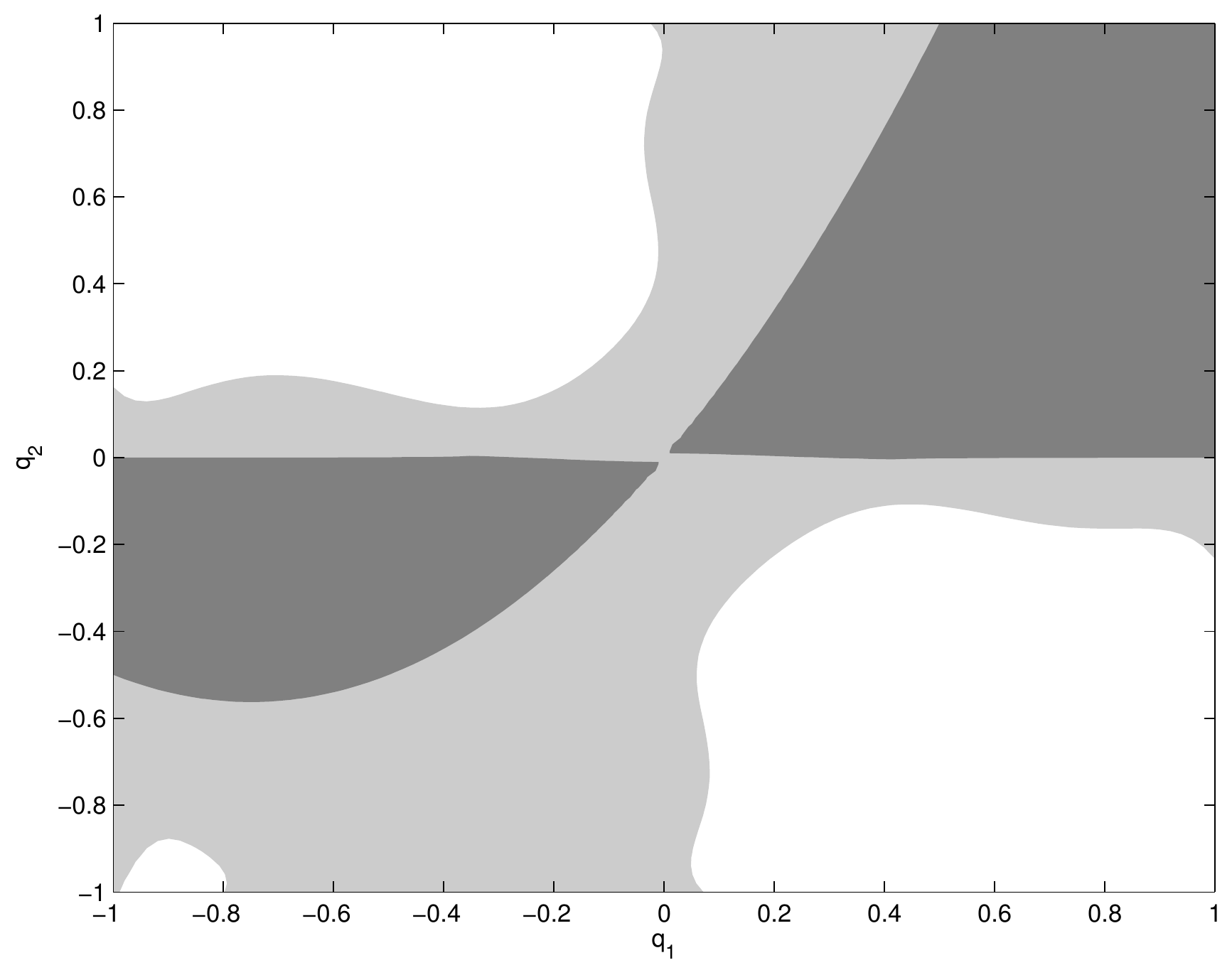}
\includegraphics[width=0.5\textwidth,height=0.3\textheight]{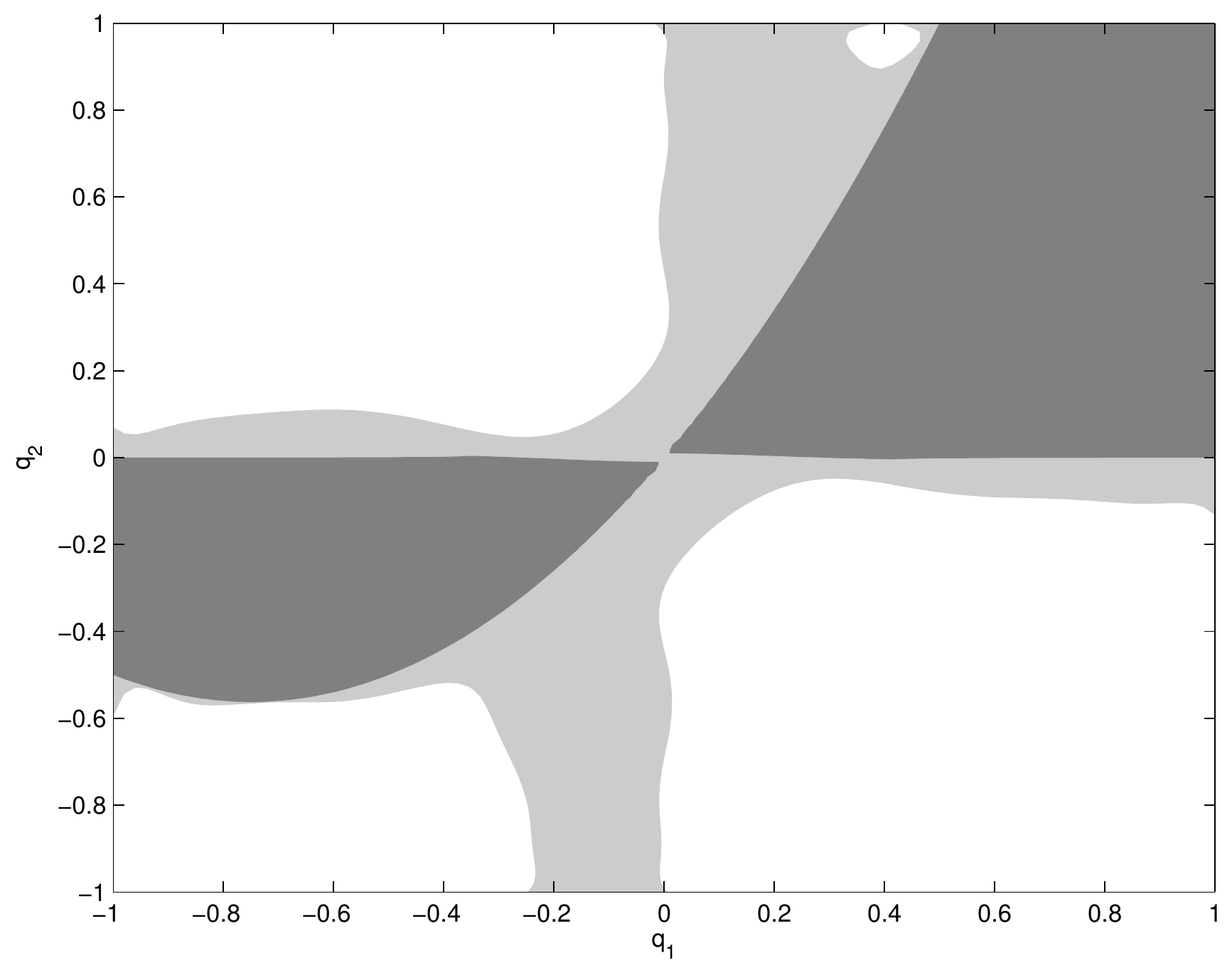}}
\caption{Stabilizability region (dark gray region) and its degree 6 outer approximation (light gray region, left) and degree 12 outer approximation (light gray region, right). Compare with Figure \ref{fig-expl3d}.}\label{fig-expl3gl}
\end{figure}

\begin{expl}\label{expl3gl}
As in Examples \ref{expl3} and \ref{expl3d} consider the polynomial
	\begin{equation*}
		p:\ s \mapsto p(q,s) = s^3 + (q_1+\tfrac{3}{2}) s^2 + q_1^2s + q_1q_2.
	\end{equation*}	
	We have $\mathcal{Z}_{r,d'} := \{(q,x,y) \in \mathcal{Z} : x - \hat{v}_{d'}(q) \geq 0\}$ with $\mathcal{Z}$ given in Example \ref{expl3}.  In Figure \ref{fig-expl3gl} we see the outer approximations of degree $d=6$ resp. $d=12$ of the stabilizability region
obtained for the choice $d'=8$. A careful examination reveals that Assumption \ref{ass2} is slightly violated here, yet this has no
effect on the validity of the zero sublevel set approximation.
Computing the degree 12 approximation takes a few minutes.
\end{expl}

\section{Conclusion}

In this paper we continued our long haul research programme consisting of developing and applying
semidefinite programming hierarchies for approximating potentially complicated objects (arising in optimization and control)
with simple objects, namely polynomials of given degrees. The complicated object of interest here
was the polynomial abscissa, which has low regularity, while being ubiquitous in linear systems control.

In section \ref{upper} we described how to construct polynomial upper approximations to the abscissa with
guarantees of $L^1$ convergence (or equivalently almost uniform convergence) on compact sets.
Constructing polynomial lower approximations with similar
convergence guarantees has proved to be much more challenging. We proposed a first approach in Section \ref{esf}
using elementary symmetric functions which is quite general but also computationally challenging due
to the introduction of many lifting variables. This motivated the study of a second approach in Section \ref{lowergl}
using the Gau{\ss}-Lucas theorem which is less computationally demanding, but unfortunately much more involved
and subject to working assumptions.

An interesting question that would deserve careful investigation is whether our $L^1$ convergence
guarantees can be strengthened to $L^{\infty}$,  i.e. to uniform convergence, since we know that
the polynomial abscissa is continuous, and hence that it can be uniformly approximated by polynomials
on compact sets. For this the semidefinite programming hierarchy
should be modified accordingly.

\end{document}